\documentclass{article}

\usepackage[latin1]{inputenc}
\usepackage[english]{babel}

\usepackage{amsmath}
\usepackage{amsthm}
\usepackage{amssymb}

\usepackage[final, notref, notcite]{showkeys}
\usepackage{color}
\usepackage{graphicx}
\usepackage{hyperref}
\hypersetup{
    colorlinks=true,
    linkcolor=black,
    citecolor=black,
    filecolor=black,
    urlcolor=black,
}
\usepackage{cleveref}
\usepackage{fullpage}

\usepackage[shortlabels]{enumitem}

\usepackage[autostyle]{csquotes}

\usepackage[
    backend=bibtex,
    style=numeric-comp,
    firstinits=true,
    maxbibnames=99
]{biblatex}
\newcommand{\id}{\mathcal{I}}

\newcommand{\N}{\mathbb{N}}
\newcommand{\R}{\mathbb{R}}
\newcommand{\C}{\mathbb{C}}
\renewcommand{\H}{\mathbb{H}}
\renewcommand{\ss}{\mathbb{S}}
\newcommand{\F}{\mathbb{F}}

\newcommand{\Q}{\mathcal{Q}}
\renewcommand{\Re}{\mathrm{Re}}
\renewcommand{\Im}{\mathrm{Im}}
\newcommand{\lhol}{\mathcal{SH}_L}
\newcommand{\rhol}{\mathcal{SH}_R}
\newcommand{\intrin}{\mathcal{N}}

\newcommand{\sderiv}[1][]{\partial_{S#1}}
\newcommand{\rsderiv}[1][]{\partial_{\overset{\leftarrow}{S#1}}}
\newcommand{\bpartial}{\overline{\partial}}
\newcommand{\rpartial}[1]{\partial_{\overset{\leftarrow}{#1}}}
\newcommand{\rbpartial}[1]{\bpartial_{\overset{\leftarrow}{#1}}}

\newcommand{\boundOP}{\mathcal{B}}
\newcommand{\closOP}{\mathcal{K}}
\newcommand{\cliff}[1]{\mathbb{R}_{#1}}

\newcommand{\ran}{\operatorname{ran}}
\newcommand{\dom}{\operatorname{dom}}
\newcommand{\dist}{\operatorname{dist}}
\newcommand{\ext}{\operatorname{ext}}
\newcommand{\fdom}{\mathcal{D}}
\newcommand{\res}{\operatorname{Res}}
\newcommand{\linspan}[1]{\operatorname{span}_{#1}}

\theoremstyle{plain}
\newtheorem{theorem}{Theorem}[section]
\newtheorem{lemma}[theorem]{Lemma}

\newtheorem{corollary}[theorem]{Corollary}

\theoremstyle{definition}
\newtheorem{definition}[theorem]{Definition}
\newtheorem{example}[theorem]{Example}

\theoremstyle{remark}
\newtheorem{remark}[theorem]{Remark}

\title{A direct approach to the $S$-functional calculus for closed operators}

\author{
Jonathan Gantner \\
Politecnico di Milano\\
Dipartimento di Matematica\\
Via E. Bonardi, 9\\
20133 Milano, Italy\\
jonathan.gantner@polimi.it\\
}

\crefname{enumi}{}{}

\begin{document}
\maketitle

\begin{abstract}
The $S$-functional calculus for slice hyperholomorphic functions generalizes the Riesz-Dunford-functional calculus for holomorphic functions to quaternionic linear operators and to $n$-tuples of noncommuting operators. For an unbounded closed operator, it is defined, as in the classical case, using an appropriate transformation and the $S$-functional calculus for bounded operators. This is however only possible  if the  $S$-resolvent set of the operator contains a real point.

In this paper, we define the $S$-functional calculus directly via a Cauchy integral, which allows us to consider also operators whose resolvent sets do not contain real points. We show that the product rule and the spectral mapping theorem also hold true with this definition and that the $S$-functional calculus is compatible with polynomials, although polynomials are not included in the set of admissible functions if the operator is unbounded. 

We also prove that the $S$-functional calculus is able to create spectral projections. In order to do this, we remove another assumption usually made: we do not assume that admissible functions are defined on slice domains. Instead, we also consider functions that are defined on not necessarily connected sets. This leads to an unexpected phenomenon: the $S$-functional calculi for left and right slice hyperholomorphic functions become inconsistent and give different operators for functions that are both left and right slice hyperholomorphic. We show that any such function is the sum of a locally constant and an intrinsic function. For intrinsic functions both functional calculi agree, but for locally constant functions they must give different operators unless any spectral projection commutes with arbitrary scalars.
 \end{abstract}
\section{Introduction}
Since \citeauthor{Birkhoff:1936} showed in their paper \cite{Birkhoff:1936} of 1936 that quantum mechanics can only be formulated using either complex numbers or quaternions, mathematicians have tried to develop a theory of quaternionic linear operators that provides techniques similar to those for complex linear operators. However, even defining such basic concepts as the spectrum or the resolvent of a quaternionic linear operator caused serious problems. Thus, little progress was made until the discovery of the $S$-functional calculus and the $S$-spectrum about ten years ago.

In the classical theory, the spectrum $\sigma(T_{\C})$ of a linear operator $T_{\C}$ is a generalization of the concept of eigenvalues. It consists of those points $\lambda\in\C$ such that the operator  $\lambda\id-T_{\C}$, which appears in the respective eigenvalue equation, does not have a bounded inverse.  

For a quaternionic right linear operator~$T$, due to the noncommutativity of the quaternionic multiplication, one can consider two possible eigenvalue equations, namely
\[ T(v) - v\lambda = 0 \qquad\text{and}\qquad T(v) - \lambda v = 0.\]
They lead to the notions of right and left eigenvalues. Right eigenvalues were useful in the mathematical theory, for instance for proving the spectral theorem for quaternionic matrices in \cite{Farenick:2003}. They also admit a physical interpretation in quaternionic quantum theory \cite{Adler:1995}. However, the operator associated to the right eigenvalue equation is not right linear and therefore not suitable for defining a resolvent operator and in turn a notion of spectrum. The operator associated to the left eigenvalue equation on the other hand is right linear, but the left eigenvalues have no obvious mathematical or physical meaning.

Another difference between the complex and the quaternionic setting is that right eigenvalues do not occur individually: if $v$ is a right eigenvector of $T$ with respect to $\lambda$, then 
\[T(va) = T(v)a = v\lambda a = va (a^{-1}\lambda a)\]
 for an arbitrary quaternion $a$. Unless $a$ and $\lambda$ commute, the vector $va$ is therefore an eigenvalue with respect to $a^{-1}\lambda a$ instead of $\lambda$.  Thus, if $\lambda$ is a right eigenvector of $T$, then any quaternion of the form $a^{-1}\lambda a$ is also a right eigenvector of $T$ and hence the set of right eigenvalues is axially symmetric. Moreover, the set of eigenvectors associated to a single eigenvalue does in general not constitute a quaternionic right linear space. 

The discovery of the $S$-spectrum the $S$-resolvent operators solved these problems. The $S$-resolvent set $\rho_S(T)$ consists of all quaternions $s$ such that the the operator $\Q_{s}(T) := T^2-2\Re(s) T + |s|^{2}\id$ has a bounded inverse defined on the entire space $V$ and the $S$-spectrum is defined as the complement of the $S$-resovlent set, that is
\[\rho_S(T) = \{s\in\H: \Q_s(T)^{-1}\in\boundOP(V)\}\quad \text{and}\quad\sigma_S(T):= \H\setminus \rho_S(T).\]
For any $s\in\rho_S(T)$, the left and right $S$-resolvent operators are defined as
\[S_L^{-1}(s,T) = \Q_s(T)\overline{s}- T\Q_s(T)^{-1}\quad\text{and}\quad S_R^{-1}(s,T) = \overline{s}\Q_s(T)^{-1}-TQ_s(T)^{-1}.\]

The introduction of these fundamental concepts several years ago opened new possibilities in quaternionic operator theory and allowed to generalize important concepts of classical operator theory: the $S$-functional calculus, which can for instance be found in the monograph \cite{Colombo:2011}, generalizes the Riesz-Dunford functional calculus for holomorphic functions to quaternionic linear operators. Even more, \citeauthor{Colombo:2008} showed in \cite{Colombo:2008} that it also applies to $n$-tuples of noncommuting operators. 

The $S$-functional calculus allows to develop the theory of quaternionic operator groups and semi-groups~\cite{Colombo:2011b, Alpay:2014}. Moreover, the $H^{\infty}$-functional calculus for quaternionic linear operators and for $n$-tuples of noncommuting operators has recently been introduced in \cite{Alpay:} and the continuous functional calculus for normal operators on a quaternionic Hilbert space was developed in \cite{Ghiloni:2013}. It has even been possible to prove the spectral theorems for unitary and for unbounded normal quaternionic linear operators in \cite{Alpay:2016a} resp. \cite{Alpay:2016}. They are also based on the $S$-spectrum. 

Finally, the $S$-point spectrum of an operator coincides with the set of its right eigenvalues  \cite{Colombo:2012}. Thus, the $S$-spectrum admits a physical interpretation within quaternionic quantum theory.

Furthermore, the above concepts do even have applications outside noncommutative operator theory. The $S$-resolvent operators appear for instance in slice hyperholomorphic Schur analysis, in particular within realizations of slice hyperholomorphic Schur functions, see \cite{Alpay:2012, Alpay:2013, Alpay:2014b}. 

In this paper, we discuss and clarify several aspects of the $S$-functional calculus for closed quaternionic operators or $n$-tuples of closed non-commuting operators. It is the analogue of the Riesz-Dunford functional calculus in these settings and it is based on the notion of slice hyperholomorphicity, a generalized notion of holomorphicity: let either $\F_0 = \H$ and $\F= \H$  or let $\F$ be the real Clifford-algebra $\cliff{n}$ and let $\F_0$ be the set of paravectors in $\cliff{n}$. Then any  $x\in\F_0$ can be written in the form $x = x_0 + I_xx_1$ such that $x_0\in\R$, $x_1>0$ and $I_x$ is an imaginary unit satisfying $I_x^2 = -1$.
A function $f:U\subset \F_0 \to \F$ is called left slice hyperholomorphic, if it is of the form 
\[f(x) = \alpha(x_0,x_1) + I_x\beta(x_0,x_1) \]
where $\alpha(x_0,-x_1) = \alpha(x_0,x_1)$ and $\beta(x_0,-x_1) = \beta(x_0,x_1)$ and $\alpha$ and $\beta$ satisfy the Cauchy-Riemann equations. 

Let now either $V$ be a two-sided quaternionic Banach space and let  $T$ be a quaternionic right linear operator on $V$ or let $V$ be a two-sided Banach-module over $\cliff{n}$ and let $T$ be an operator of paravector type on $V$.

Assume that $T$ is bounded and let $f$ be left slice hyperholomorphic on a suitable domain $U$ (precisely, an axially symmetric slice domain) that contains $\sigma_S(T)$ and has a sufficiently regular boundary. Then one defines, inspired by the Cauchy formula for left slice hyperholomorphic functions, 
\begin{equation}\label{Int1}
f(T) := \frac{1}{2\pi} \int_{\partial(U\cap\C_I)}S_L^{-1}(s,T)\,ds_I\,f(s),
\end{equation}
where  $I$ is an arbitrary imaginary unit, $ds_I = (-I) ds$ and $\partial(U\cap\C_I)$ is the boundary of $U$ in the complex plane spanned by $1$ and $I$. This integral is independent of the choices of $U$ and $I$. 

In the complex setting, one can define the Riesz-Dunford functional calculus for closed unbounded operators quite efficiently by reducing it to the one for bounded operators \cite{Dunford:1988}. An analogous approach is possible in our settings, if $T$ is closed and $\rho_S(T)$ contains a real point $a$, see \cite{Colombo:2011}.  Then the operator $A :=  (T-a\id)^{-1}$ equals $-S_{L}^{-1}(a,T)$ and is therefore bounded. If one sets $\Phi_a(s) = (s-a)^{-1}$, then $f \mapsto f \circ\Phi_a^{-1}$  defines a bijective relation between those functions that are slice hyperholomorphic on the $S$-spectrum  of $T$ and at infinity and those functions that are slice hyperholomorphic on the $S$-spectrum of $A$. One therefore defines 
\begin{equation*}f(T) = (f\circ\Phi_a^{-1})(A).\end{equation*}
If $U$ is an axially symmetric slice domain with sufficiently regular boundary and $\sigma_S(T)\subset U$ such that $f$ is slice hyperholomorphic on $U$, then the operator $f(T)$ can  be represented as
\begin{equation} \label{repUno}
f(T) = f(\infty)\id + \frac{1}{2\pi}\int_{\partial(U\cap\C_I)}S_L^{-1}(s,T)\,ds_I\,f(s),
\end{equation}
where the notation is as in \eqref{Int1}. In particular this shows, that $f(T)$ is independent of the choice of the real point $a$. Similarly, one can proceed to define the $S$-functional calculus for right slice hyperholomorphic functions. However, this procedure requires that the $S$-spectrum of the operator $T$ contains a real point. Otherwise, for non-real $a$, the operator $(T-a\id)^{-1}$ does not equal the $S$-resolvents and the composition with the function $\Phi_a^{-1}$ does not preserve slice hyperholomorphicity.
 
 In this paper we choose a different approach and define the $S$-functional calculus directly via \eqref{repUno} as it is done in \cite{Taylor:1951} for complex linear operators. We show that this defines a meaningful functional calculus under the only assumption that the $S$-resolvent set of $T$ is nonempty, which is analogue to the classical complex case. In particular, we do not require that $\rho_S(T)$ contains a real point.
 
 We also remove another assumption, which existed only due to the history of the development of the theory, namely that the function $f$ is defined on a slice domain. Instead, we consider functions  that are defined on not necessarily connected sets. This is in particular important for obtaining spectral projections within the scope of the $S$-functional calculus, but it leads to an unexpected phenomenon: the $S$-functional calculi for left and right slice hyperholomorphic functions become inconsistent and must therefore be considered as separate calculi. Indeed, for functions that are left and right slice hyperholomorphic,  $f(T)$ will in general be different depending on whether $f$ is considered as a left or as a right slice hyperholomorphic function.  
 
 We discuss this phenomenon in detail: it turns out that, up to addition with a locally constant function, every left and right slice hyperholomorphic function is intrinsic. For intrinsic functions and for functions whose domains are connected, both functional calculi agree. For a locally constant function $f$ however, the two $S$-functional calculi must give different operators $f(T)$, unless all invariant subspaces associated to spectral projections are two-sided subspaces.  Finally, we also give an explicit example for this situation in two dimensions. 
 
Furthermore, we prove the product rule using the $S$-resolvent equation, which was recently discovered in \cite{Alpay:2015}, and we discuss the compatibility of the $S$-functional calculus with polynomials in $T$. For bounded operators, these are included in the admissible class of functions, but for unbounded operators they need to be investigated separately. As it happens often in the slice hyperholomorphic setting, we have to restrict ourselves to intrinsic functions. We show that if $x\in\dom(T^n)$ for some $n\in\N$, then $f(T)x\in\dom(T^n)$ and $f(T)p(T)x = p(T)f(T)x$ for any admissible intrinsic function $f$ and any intrinsic polynomial $p$ of degree lower or equal to $n$. Moreover, if $f$ has a zero of order $n$ at infinity, then $f(T)x \in \dom(T^n)$ for any vector $x$ and $p(T)f(T)x = (pf)(T)x$ for any intrinsic polynomial $p$ of degree lower or equal to $n$. As a consequence, we obtain that $p(T)$ is closed for any intrinsic polynomial $p$.

Finally, we  show that the spectral mapping theorem and the theorem on composite functions also hold with our definition of the $S$-functional calculus and we discuss the possibility to generate spectral projections: if $\sigma$ is a spectral set of $T$, that is an open and closed subset of the closure of $\sigma_S(T)$ in $\F_0\cup\{+\infty\}$ and $U_{\sigma}$ is an axially symmetric open set such that $\sigma \subset U$ and $\sigma \cap (\sigma_S(T)\setminus\sigma) = \emptyset$, then the characteristic function $\chi_{\sigma}(s)$ of $U_{\sigma}$ is admissible for the $S$-functional calculus and $E_{\sigma}:=\chi_{\sigma}(T)$ is a projection that commutes with $T$. Spectral projections were studied for bounded operators in \cite{Colombo:2011, Alpay:2015}, but they have never been investigated for unbounded quaternionic operators.

\section{Preliminary results}
Slice hyperholomorphic functions can be considered in different settings. We shall be interested in two different situations, in which they allow to develop, with analogous arguments, the spectral theory of certain operators: the quaternionic setting and the setting of Clifford-algebra-valued functions of a paravector variable. Both settings are treated in detail in \cite{Colombo:2011}.

The skew-field of quaternions consists of the real vector space 
$$
\H:=\left\{\xi_0e_0 + \sum_{\ell=1}^3\xi_\ell e_\ell: \xi_\ell\in\R\right\},
$$
endowed with an associative product with unity $e_0$ that satisfies
\[e_ie_j = -e_je_i\qquad\text{and}\qquad e_i^2 = -1\]
for $i,j\in\{1,\ldots,3\}$ with $i\neq j$. Since $e_0$ is the unity of the skew-field, we write $1$ instead of $e_0$ and identify $\linspan{\R}\{e_0\}$ with the field of real numbers. The real part of a quaternion $x = \xi_0 + \sum_{\ell=1}^3\xi_\ell e_\ell$ is therefore defined as $\Re(x) := \xi_0$, its vectorial part as $\underline{x} := \sum_{\ell=1}^3\xi_\ell e_\ell$ (sometimes we also write $\Im(x) := \underline{x}$) and its conjugate as $\overline{x} := \Re(x) - \underline{x}$. The modulus of $x$ is defined by $|x|^2 = \sum_{i=0}^3 |\xi_i|^2$ and the relation $x\overline{x} = \overline{x}x = |x|^2$ holds true. Hence $x^{-1} = \overline{x}/|x|^2$. 

Each element of the set
\[\ss := \{ x\in\H: \Re(x) = 0, |x| = 1 \}\]
 is a square-root of $-1$ and is therefore called an imaginary unit. For any $I\in\ss$, the subspace $\C_I := \{x_0 + I x_1: x_0,x_1\in\R\}$ is an isomorphic copy of the field of complex numbers. If $I,J\in\ss$ with $I\perp J$, set $K=IJ = -JI$. Then $1$, $I$, $J$ and $K$ form an
 orthonormal basis of $\H$ as a real vector space and $1$ and $J$ form an orthonormal basis of $\H$ as a left or right vector space over the complex plane $\C_I$, that is
 \begin{equation} \label{QuatSplit}
 \H = \C_I + \C_I J\quad\text{and}\quad \H = \C_I + J\C_I.
\end{equation}
For consistency with the Clifford-algebra setting, we introduce the following notation: after choosing $I$ and $J$, we set $I_{\emptyset}=1$, $I_{\{1\}} = I$, $I_{\{2\}}=J$ and $I_{\{1,2\}}=K=IJ$ and we call $\{I,J\}$ a generating basis of $\H$. Any quaternion can then be represented as $x = \sum_{A\subset\{1,2\}}\xi_AI_A$ with $\xi_A\in\R$ and as $x = \sum_{A\subset\{2\}} z_AI_A$ or $x=\sum_{A\subset\{2\}}I_A\widetilde{z_A}$ with $z_{A}, \widetilde{z_A}\in\C_I$.

  Finally, any quaternion $x$ belongs to a complex plane: if we set
 \[I_x := \begin{cases}\underline{x}/|\underline{x}|,& \text{if  }\underline{x} \neq 0 \\ \text{any }I\in\ss, \quad&\text{if }\underline{x}  = 0\end{cases},\]
 then $x = x_0 + I_x x_1$ with $x_0 =\Re(x)$ and $x_1 = |\underline{x}|$. The set
 \[
 [x] := \{x_0 + Ix_1: I\in\ss\},
 \]
is a 2-sphere, that reduces to a single point if $x$ is real.

In order to introduce the real Clifford-algebra $\cliff{n}$ over $n$ units, we consider the space
\[\R^{n+1} = \left\{ \xi_0e_0 + \sum_{i=1}^n\xi_ie_i: \xi_i\in\R \right\}.\]
The Clifford-algebra $\cliff{n}$ is the algebra that is generated by this space when  $e_0$ is the unity of the multiplication and
\[ e_ie_j = -e_je_i\qquad\text{and}\qquad e_i ^2 = -1\]
for  $i,j\in\{1,\ldots,n\}$ with $i\neq j$. Any element $x\in\cliff{n}$ can be written in the form
$x = \sum_{A\subset \{1,\ldots,n\}}\xi_Ae_A$ with $\xi_A\in\R$, where $e_{\emptyset} = e_0$ and $e_A = e_{i_1}\cdots e_{i_k}$ for any other subset $A =\{i_1<i_2<\ldots<i_k\}$  of $\{1,\ldots,n\}$. Since $e_0$ is the unity of the algebra, we write $1$ instead of $e_0$ and identify $\R$ with $\linspan{\R}\{e_0\}$. The conjugation on the Clifford  algebra $\cliff{n}$ is defined by the rule $\overline{xy} = \overline{y}\;\!\overline{x}$  and its action on the generating basis
$\overline{e_0} = e_{0}$ and $\overline{e_i} = - e_i$ for $i\in\{1,\ldots,n\}$. The modulus of an element $x=\sum_{A\subset\{1,\ldots,n\}}\xi_Ae_A\in\cliff{n}$ is defined via $|x|^2 = \sum_{A\subset\{1,\ldots,n\}}\xi_A^2$.

An element $x = \xi_0 + \sum_{i=1}^n\xi_ie_i$ of the generating space $\R^{n+1}$ is called a paravector. A paravector can be decomposed into its real part $\Re(x) = \xi_0$ and its vectorial part $\underline{x} = \sum_{i=1}^n\xi_ie_i$, which will sometimes also be denoted as $\Im(x) = \underline{x}$. Obviously, $x = \Re(x) + \underline{x}$ and $\overline{x} = \Re(x) - \underline{x}$. Any paravector is invertible as $x\overline{x} = |x|^2$ and hence $x^{-1} = \overline{x}/|x|^2$. However, not any element of $\cliff{n}$ is invertible: $\cliff{1}$ is isomorphic to the field of complex numbers $\C$ and $\cliff{2}$ is isomorphic to the skew-field of quaternions $\H$, but if $n\geq 3$, then $\cliff{n}$ contains zero divisors. 

In the Clifford-algebra setting, imaginary units are elements of the set
\[ \ss := \{x\in\R^{n+1}: \Re(x) = 0, |x|=1 \}.\]
Also here $I^2=-1$ for any $I\in\ss$. Hence $\C_I := \{x_0 + Ix_1: x_0,x_1\in\R\}$  is an isomorphic copy of the field of complex numbers. Any paravector $x$ belongs to such a complex plane: if we set 
 \[I_x := \begin{cases}\underline{x}/|\underline{x}|,& \text{if  }\underline{x} \neq 0 \\ \text{any }I\in\ss, \quad&\text{if }\underline{x}  = 0,\end{cases}\]
 then $x = x_0 + I_x x_1$ with $x_0 =\Re(x)$ and $x_1 = |\underline{x}|$. The set
 \[
 [x] := \{x_0 + Ix_1: I\in\ss\},
 \]
is an $(n-1)$-sphere in $\R^{n+1}$, that reduces to a single point if $x$ is real.

Moreover, for any $I\in\ss$, one can find a basis of $\R^{n+1}$ that contains $I$ and generates $\cliff{n}$ as an algebra, cf. \cite[Proposition~2.2.10.]{Colombo:2011}.
\begin{lemma}\label{CliffSplit}
Let $I\in\ss$ and set $I_1 := I$. Then there exist $I_2, \ldots, I_n\in\ss$ such that $\{1,I_1,\ldots,I_2\}$ forms an generating basis of $\cliff{n}$, i.e. $I_iI_j = -I_jI_i$. In this case any $x\in\cliff{n}$ can be written in the form $x = \sum_{A\subset\{1,\ldots,n\}}\xi_aI_A$ with $\xi_A\in\R$, where $I_{\emptyset}=1$ and $I_{A} = I_{i_1}\cdots I_{i_k}$ for any other subset $A=\{i_1<\ldots <i_k\}$ of $\{1,\ldots,n\}$. We thus have
\[x = \sum_{A\subset\{2,\ldots,n\}} z_A I_A\quad\text{ with }z_A \in\C_I\]
and similarly
\[x = \sum_{A\subset\{2,\ldots,n\}}  I_A \widetilde{z_A}\quad\text{ with }\widetilde{z_A}\in\C_I.\]
In particular, $\cliff{n}$ can be considered as a complex vector space over $\C_I$, when the multiplication with scalars is simply defined by restricting the Clifford-multiplication from the left or from the right to elements of $\C_I$.
\end{lemma}
As pointed out before, the theory of slice hyperholomorphic functions allows to develop the spectral theory of operators in two different settings with analogue arguments. We introduce a notation to cover both cases simultaneously.
\begin{definition}
In the quaternionic setting, we set $\F_0:= \H$ and $\F:=\H$. In the Clifford-algebra setting, we denote $\F_0:= \R^{n+1}$ and $\F:= \cliff{n}$.
\end{definition}

\subsection{Slice hyperholomorphic functions} 
The notion of slice hyperholomorphicity is the generalization of holomorphicity to quaternion- resp. $\cliff{n}$-valued functions that underlies the theory of linear operators in these settings. 
We recall the main results on slice hyperholomorphic functions. Their proofs can be found in \cite{Colombo:2010c} and \cite{Colombo:2011}.
 
\begin{definition}
A set $U\subset\F_0$ is called
\begin{enumerate}[(i)]
\item axially symmetric if $[x]\subset U$ for any $x\in U$ and
\item a slice domain if $U$ is open, $U\cap\R\neq 0$ and $U\cap\C_I$ is a domain for any $I\in\ss$.
\end{enumerate}
\end{definition}

\begin{definition}\label{SHol}
Let $U\subset\F_0$ be an axially symmetric open set. A function $f: U\to \F$ is called a left slice function, if it has the form
\begin{equation}\label{LSHol}
 f(x) = \alpha(x_0,x_1) + I_x\beta(x_0,x_1),\quad \forall x = x_0 + I_x x_1 \in U\end{equation}
such that the functions $\alpha$ and $\beta$ satisfy the compatibility condition 
\begin{equation}\label{SymCond}
\alpha(x_0,x_1) = \alpha(x_0,x_1)\qquad \beta(x_0,x_1) = - \beta(x_0,-x_1).
\end{equation}
A real differentiable left slice function $f: U\to\F$ is called left slice hyperholomorphic if $\alpha$ and $\beta$ satisfy the Cauchy-Riemann-differential equations
\begin{equation}\label{CR}\begin{split}
\frac{\partial}{\partial x_0} \alpha(x_0,x_1) &= \frac{\partial}{\partial x_1} \beta(x_0,x_1)\\ 
\frac{\partial}{\partial x_0}\beta(x_0,x_1) &= -\frac{\partial}{\partial x_1}\alpha(x_0,x_1).
\end{split}\end{equation}
A function $f:U\to\F$ is called right slice function if it has the form 
\begin{equation}\label{RSHol}
 f(x) = \alpha(x_0,x_1) + \beta(x_0,x_1)I_x,\quad \forall x = x_0 + I_x x_1 \in U,
 \end{equation}
such that the functions $\alpha$ and $\beta$ satisfy \eqref{SymCond}. If in addition $f$ is real differentiable and $\alpha$ and $\beta$ satisfy \eqref{CR}, then $f$ is called right slice hyperholomorphic.

The sets of left and right slice hyperholomorphic functions on $U$ are denoted by $\lhol(U)$ and $\rhol(U)$, respectively. Finally, we say that a function $f$ is left or right slice hyperholomorphic on a closed axially symmetric set $K$, if there exists an open axially symmetric set $U$ with $K\subset U$ such that $f\in \lhol(U)$ resp. $\rhol(U)$. 
\end{definition}

\begin{corollary}
Let $U\subset\F_0$ be axially symmetric.
\begin{enumerate}[(i)]
\item If $f,g\in\lhol(U)$ and $a\in\F$, then $fa+g\in\lhol(U)$.
\item If $f,g\in\rhol(U)$ and $a\in\F$, then $af + g\in\rhol(U)$. 
\end{enumerate}
\end{corollary}
On axially symmetric slice domains, slice hyperholomorphic functions can be characterized as those functions that lie in the kernel of a slicewise Cauchy-Riemann-operator. As a consequence, the restriction of a slice hyperholomorphic function to a complex plane can be split into holomorphic components.
\begin{definition}
Let $U\subset\F_0$ be open. For a real differentiable function $f:U\to\F$, we denote $f_I := f|_{\C_I}$ for $I\in\ss$ and define the following differential operators: for $x = x_0 + I_x x_1$ we set
\begin{align*}
\partial_I f(x) &= \frac{\partial}{\partial x_0} f_{I_x}(x) - I_x \frac{\partial}{\partial x_1} f_{I_x}(x) \\
\bpartial_I f(x) &= \frac{\partial}{\partial x_0} f_{I_x}(x) + I_x \frac{\partial}{\partial x_1} f_{I_x}(x)
\end{align*}
and
\begin{align*}
f(x)\rpartial{I} &= \frac{\partial}{\partial x_0} f_I(x) - \frac{\partial}{\partial x_1} f_I(x)I_x\\
 f(x)\rbpartial{I} &= \frac{\partial}{\partial x_0} f_I(x) +  \frac{\partial}{\partial x_1} f_I(x)I_x.
\end{align*}
The arrow $\leftarrow$ indicates that the operators $\rbpartial{I}$ and $\rpartial{I}$ act from the right.
\end{definition}
\begin{corollary}\label{IDeriv}
Let $U\subset\F_0$ be open and axially symmetric.
\begin{enumerate}[(i)]
\item If $f\in\lhol(U)$, then $\bpartial_I f = 0$. If $U$ is an axially symmetric slice domain, then $f\in\lhol(U)$ if and only if $\bpartial_I f = 0$
\item  If $f\in\rhol(U)$, then $f\rbpartial{I} = 0$.  If $U$ is an axially symmetric slice domain, then $f\in\rhol(U)$ if and only if $f \rbpartial{I}  = 0$. 
\end{enumerate}
\end{corollary}
\Cref{IDeriv} states that a function is left resp. right slice hyperholomorphic if any restriction to a complex subplane $\C_I$ is a holomorphic function with values in the left resp. right Banach space $\F$ over $\C_I$. Splitting into components with respect to a chosen basis as in \Cref{CliffSplit} immediately yields the next result.
\begin{lemma}[Splitting Lemma]\label{SplitLem}
Let $U\subset\F_0$ be axially symmetric, let $I\in\ss$ and let $I_2,\ldots,I_n$ be imaginary units that form, together with $I$, a generating basis of $\F$.
\begin{enumerate}[(i)]
\item If $f\in\lhol(U)$, then there exist holomorphic functions $f_A: U\cap\C_I\to\C_I$ such that $ f_I = \sum_{A\subset\{2,\ldots,n\}}f_AI_A$.
\item If $f\in\rhol(U)$, then there exist holomorphic functions $f_A: U\cap\C_I\to\C_I$ such that $ f_I = \sum_{A\subset\{2,\ldots,n\}}I_Af_A$.
\end{enumerate}
\end{lemma}
\begin{remark}\label{History}
Originally, slice hyperholomorphic functions were defined as functions that satisfy $\bpartial_I f = 0$ resp. $f\rbpartial{I}  = 0$. In principle this leads to a larger class of functions, but on axially symmetric slice domains both definitions are equivalent. Indeed, in this case the representation formula, cf. \Cref{RepFo}, holds true and allows for a representation of the form \eqref{LSHol} resp. \eqref{RSHol}. Therefore, the theory of slice hyperholomorphicity with the original definition was only developed for functions that are defined on axially symmetric slice domains. 

However, most results on slice hyperholomorphic functions actually require the possibility of a representation of the form \eqref{LSHol} resp. \eqref{RSHol} and not that the function is defined on an axially symmetric slice domain. Hence, \Cref{SHol} seems to be more appropriate since it allows to extend the theory to functions defined on sets that are not connected or do not intersect the real line. This is in particular important in spectral theory, for instance when one wants to define spectral projections.
\end{remark}

\begin{definition}
Let $U\subset\F_0$ be axially symmetric. A left slice hyperholomorphic function $f(x) = \alpha(x_0,x_1) + I_x\beta(x_0,x_1)$ is called intrinsic if $\alpha$ and $\beta$ are real-valued. We denote the set of intrinsic functions on $U$ by $\intrin(U)$.
\end{definition}
Note that intrinsic functions are both left and right slice hyperholomorphic because $\beta(x_0,x_1)$ commutes with the imaginary unit $I_x$. The converse is not true: the constant function $x\mapsto b\in\F\setminus\R$ is left and right slice hyperholomorphic, but it is not intrinsic.
\begin{corollary}\label{IntChar}
Let $f$ be left or right slice hyperholomorphic on $U$. The following statements are equivalent.
\begin{enumerate}[(i)]
\item $f$ is intrinsic.
\item $f(\overline{x}) = \overline{f(x)}$ for all $x\in U$.
\item $f(U\cap\C_I)\subset \C_I$ for all $I\in\ss$.
\end{enumerate}
\end{corollary}

The importance of the class of intrinsic functions is due to the fact that multiplication and composition with intrinsic functions preserve slice hyperholomorphicity, which is not true for arbitrary slice hyperholomorphic functions.
\begin{corollary} Let $U\subset\F_0$ be axially symmetric.
\begin{enumerate}[(i)]
\item If $f\in\intrin(U)$ and $g\in\lhol(U)$, then $fg\in\lhol(U)$. If $f\in\rhol(U)$ and $g\in\intrin(U)$, then $fg\in\rhol(U)$.
\item If $g\in\intrin(U)$ and $f\in\lhol(g(U))$, then $f\circ g\in \lhol(U)$. If $g\in\intrin(U)$ and $f\in\rhol(g(U))$, then $f\circ g\in \rhol(U)$.
\end{enumerate}
\end{corollary}

Important examples of slice hyperholomorphic functions are power series with coefficients in $\F$: series of the form $\sum_{n=0}^{+\infty}x^na_n$ are left slice hyperholomorphic and series of the form $\sum_{n=0}^{\infty} a_nx^n$ are right slice hyperholomorphic on their domain of convergence. A power series is intrinsic if and only if its coefficients are real.

Conversely, any slice hyperholomorphic function can be expanded into a power series of the respective type, but only at real points.
\begin{definition}
The slice-derivative of a function $f\in\lhol(U)$ is defined as
\[\sderiv  f(x) = \lim_{\C_{I_x}\ni s\to x} (s-x)^{-1}(f(s)-f(x)), \]
where $\lim_{\C_{i_x}\ni s\to x} g(s)$ denotes the limit as $s$ tends to $x = x_0 + I_xx_1 \in U$ in $\C_{I_x}$.
The slice-derivative of a function $f\in\rhol(U)$ is defined as
\[  f\rsderiv(x) = \lim_{\C_{I_x}\ni s\to x} (f(s)-f(x))(s-x)^{-1}. \]
\end{definition}

\begin{corollary}
The slice derivative of a left (or right) slice hyperholomorphic function is again left (or right) slice hyperholomorphic. Moreover, it coincides with the derivative with respect to the real part, that is
\[\sderiv  f(x) = \frac{\partial}{\partial x_0} f(x) \quad \text{resp.}\quad f\rsderiv(x) = \frac{\partial}{\partial x_0} f(x). \]
\end{corollary}
\begin{theorem}
If $f$ is left slice hyperholomorphic on the ball $B_r(\alpha)$ with radius $r$ centered at $\alpha\in\R$, then
\[f(x) = \sum_{n=0}^{+\infty} (x-\alpha)^n \frac{1}{n!}\sderiv ^n f(\alpha)\quad\text{for }x\in B_r(\alpha).\]
If $f$ is right slice hyperholomorphic on $B_r(\alpha)$, then
\[f(x) = \sum_{n=0}^{+\infty}\frac{1}{n!} f\rsderiv ^n(\alpha) (x-\alpha)^n \quad\text{for }x\in B_r(\alpha).\]
\end{theorem}

As a consequence of the slice structure of slice hyperholomorphic functions, their values are uniquely determined by their values on one complex plane. Consequently, any function that is holomorphic on a suitable subset of a complex plane possesses a unique slice hyperholomorphic extension.
\begin{theorem}[Representation Formula]\label{RepFo}
Let $U\subset\F_0$ be axially symmetric and let $I\in\ss$. For any $x = x_0 + I_x x_1\in U$ set $x_I := x_0 + Ix_1$. If $f\in\lhol(U)$. Then
\[f(x) = \frac12(1-I_xI)f(x_I) + \frac12(1+I_xI)f(\overline{x_I}).\]
If $f\in\rhol(U)$, then
\[f(x) = f(x_I)(1-II_x)\frac12 + f(\overline{x_I})(1+II_x)\frac12.\]
\end{theorem}

\begin{corollary}\label{extLem}
Let $I\in\ss$ and let $f:O\to\F$ be real differentiable, where $O$ is a domain in $\C_I$ that is symmetric with respect to the real axis.
\begin{enumerate}[(i)]
\item The axially symmetric hull $[O]: = \bigcup_{z\in O}[z]$ of $O$ is an axially symmetric slice domain.
\item If $f$ satisfies $\frac{\partial}{\partial x_0}f + I \frac{\partial}{\partial x_1} f= 0$, then there exists a unique left slice hyperholomorphic extension $\ext_L(f)$ of $f$ to $[O]$.
\item If $f$ satisfies $\frac{\partial}{\partial x_0}f +  \frac{\partial}{\partial x_1}fI = 0$, then there exists a unique right slice hyperholomorphic extension $\ext_R(f)$ of $f$ to $[O]$.
\end{enumerate}
\end{corollary}
\begin{remark}
If $f$ has a left and a right slice hyperholomorphic extension, they do not necessarily coincide. Consider for instance the function $z\mapsto bz$ on $\C_I$ with a constant $b\in\C_I\setminus\R$. Its left slice hyperholomorphic extension to $\F_0$ is $x\mapsto xb$, but its right
slice hyperholomorphic extension is $x\mapsto bx$.
\end{remark}
Finally, slice hy\-per\-ho\-lo\-mor\-phic functions satisfy an adapted version of Cauchy's integral theorem and a Cauchy-type integral formula with a modified kernel. 

\begin{definition}
We define the  left slice hyperholomorphic Cauchy kernel as
\[S_L^{-1}(s,x) = -(x^2-2\Re(s)x + |s|^2)^{-1}(x-\overline{s})\quad\text{for }x\in\F_0\setminus[s]\]
and the right slice hyperholomorphic Cauchy kernel as
\[S_R^{-1}(s,x) = -(x-\overline{s})(x^2-2\Re(s)x + |s|^2)^{-1}\quad\text{for }x\in\F_0\setminus[s].\]
\end{definition}
\begin{corollary}\label{SProps}
The left slice hyperholomorphic Cauchy-kernel $S_L^{-1}(s,x)$ is left slice hyperholomorphic in the variable $x$ and right slice hy\-per\-ho\-lo\-mor\-phic in the variable $s$ on its domain of definition. Moreover, we have \begin{equation}\label{SID}
S_R^{-1}(s,x) = - S_L^{-1}(x,s).
\end{equation}
\end{corollary}
\begin{remark}
If $x$ and $s$ belong to the same complex plane, they commute and the slice hyperholomorphic Cauchy-kernels reduce to the classical one:
\[ \frac{1}{s-x} = S_L^{-1}(s,x) = S_R^{-1}(s,x).\]
\end{remark}

\begin{definition}
An axially symmetric set $U\subset\F_0$ is called slice Cauchy domain if $U\cap\C_I$ is a Cauchy domain for any $I\in\ss$. More precisely, for any $I\in\ss$, the following conditions must hold:
\begin{enumerate}[(i)]
\item $U\cap\C_I$ is open
\item $U\cap\C_I$ has a finite number of components (i.e. maximal connected subsets), the closures of any two of which are disjoint
\item the boundary of $U\cap\C_I$ consists of a finite positive number of closed piecewise continuously differentiable Jordan curves, no two of which intersect.
\end{enumerate}
\end{definition}
\begin{remark}
Observe that any Cauchy domain has at most one unbounded component. This component must then contain a neighborhood of infinity.
\end{remark}

\begin{definition}
Let $I\in\ss$ and $\gamma:[0,1]\to \C_I$ be a path. We set $ds_I := (-I)ds$ such that $\int_{\gamma}f(s)\,ds_I\,g(s) := \int_{\gamma}f(s)(-I )\,ds\, g(s)$.
\end{definition}

\begin{theorem}[Cauchy's integral theorem]\label{CauchyThm}
Let $U\subset\F_0$ be a slice Cauchy domain, let $f\in\rhol(\overline{U})$ and $g\in\lhol(\overline{U})$. For any $I\in\ss$ holds
\[\int_{\partial (U\cap\C_I)}f(s)\,ds_I\,g(s) = 0.\]
\end{theorem}

\begin{theorem}[Cauchy's integral formula]\label{Cauchy}
Let $U\subset\F_0$ be a slice Cauchy domain and let $I\in\ss$. If $f\in\lhol(\overline{U})$, then
\[f(x) = \frac{1}{2\pi}\int_{\partial(U\cap\C_I)} S_L^{-1}(s,x)\,ds_I\, f(s)\quad\text{for all }x\in U.\]
If $f\in\rhol(\overline{U})$, then
\[f(x) = \frac{1}{2\pi}\int_{\partial(U\cap\C_I)}f(s)\, ds_I\, S_R^{-1}(s,x)\quad\text{for all }x\in U.\]
\end{theorem}
\begin{remark}\label{OPVal}
Although the presented results were stated for scalar valued functions, they also hold true for functions with values in a two-sided quaternionic Banach space resp. in a two-sided Banach module over $\cliff{n}$. Indeed, one can proof them in these settings with the usual technique of reducing the vector-valued to the scalar-valued case by applying elements of the dual space. For details see \cite{Alpay:2015a}.
\end{remark}

\subsection{The $S$-resolvent operator and the $S$-functional calculus}
In this paper, $V$ denotes either a two-sided quaternionic Banach space  or a two-sided Banach module over~$\cliff{n}$:
\begin{enumerate}[(i)]
\item A two-sided quaternionic Banach space is a two-sided quaternionic vector space\footnote{We shall occasionally refer to a quaternionic vector space  as a module over $\H$ in order combine the discussions of the quaternionic and the Clifford setting.} endowed with a norm $\|\cdot\|_V$ such that it is a Banach space over $\R$ and such that $|a|\|v\|_V = \|av\|_V$ and $\|va\|_V = \|v \|_V|a|$ for all $v\in V$ and all $a\in\H$. A mapping $T:\dom(T)\to V$ defined on a right subspace $\dom(T)$ of $V$ is said to be right linear if $T(ua + v) = T(u)a+ T(v)$ for all $u,v\in \dom(T)$ and all $a\in\H$. It is said to be bounded if $\|T\|_{\boundOP(V)} := \sup_{\|v\|_V = 1} \|T(v)\|_V<+\infty$. The set of all bounded right linear operators, which we denote by $\boundOP(V)$, is a two-sided quaternionic Banach space,  when it is endowed with the scalar multiplications $(aT)(v) := a(T(v))$ and $(Ta)(v) = T(av)$ and with the operator norm.

A right linear operator $T:\dom(T)\subset V\to V$ is called closed if its graph is closed and we denote the set of all such operators by $\closOP(V)$.

\item Let $V_{\R}$ be a real Banach space. Then $V = V_{\R}\otimes\cliff{n}$ is a two-sided Clifford module. If $(e_A)_{A\subset \{1,\ldots,n\}}$ is a basis of $\cliff{n}$, then any element of $V$ is of the form $v = \sum_{A\subset \{1,\dots,n\}} v_A\otimes e_A$ with $v_A\in V_{\R}$. The multiplications with scalars on the left and on the right are defined as 
\[av := \sum_{A,B\subset\{1,\ldots,n\}} (a_Bv_A)\otimes(e_Be_A)\qquad \text{and}\qquad va := \sum_{A,B\subset\{1,\ldots,n\}} (a_Bv_A)\otimes(e_Ae_B)\]
for $v\in V$ and $a = \sum_{B\subset\{1,\ldots,n\}}a_Be_B$. In the following, we shall omit the symbol $\otimes$. 

The Clifford module $V$ turns into a Banach module over $\cliff{n}$ when it is endowed with the norm $\|v\|_V := \sum_{A\subset\{1,\ldots,n\}}\|v_A\|_{V_{\R}}$, i.e. it is a real Banach space and there exists a constant $C>0$ such that $\|v a\|_V \leq C\|v\|_V|a|$ and $\|av\|_V \leq C|a|\|v\|_V$ for all $v\in V$ and all $a\in \cliff{n}$.

If $T_A, A\subset\{1,\ldots n\}$, are bounded operators on $V_{\R}$, then $T = \sum_{A\subset\{1,\ldots,n\}}T_A e_A$ is the right $\cliff{n}$-linear operator that acts as
\[T(v) = \sum_{A\subset\{1,\ldots,n\}} T_A(v_B)e_Ae_B\]
for $v = \sum_{B\subset \{1,\ldots,n\}}v_Be_B\in V$. The set of all such operators corresponds to $\boundOP(V_{\R})\otimes\cliff{n}$, which is a two-sided Banach module over $\cliff{n}$ with the norm\footnote{In the following, we omit all subscripts of norms unless it is unclear which norm we are referring to.} $\|T\|_{\boundOP(V_{\R})\otimes\cliff{n}} := \sum_{A\subset \{1,\ldots,n\}} \|T_A\|_{\boundOP(V_{\R})}$. However, in the Clifford-setting we shall only be interested in operators of paravector type, that is operators of the form
\begin{equation}\label{PVOP}T = T_0 +\sum_{i=0}^n T_ie_i\quad\text{with } T_i\in\boundOP(V_{\R})\end{equation}
and $\boundOP(V)$ shall thus denote the set of all such operators on $V$.

Similarly, we denote by $\closOP(V)$ the set of all closed operators of paravector type, i.e. operators of the form \eqref{PVOP} such that the components $T_i, i=0,\ldots,n$ are closed operators on $V_{\R}$. The operator $T$ is then defined on the common domain $\dom(T) = \bigcap_{i=0}^n \dom(T_i)$. 

\end{enumerate}
Based on the theory of slice hyperholomorphic functions, it is possible to define a functional calculus for operators as they were introduced above. It is the natural generalization of the Riesz-Dunford-functional calculus for complex linear operators to the quaternionic or Clifford-setting; for details see again \cite{Colombo:2011}.

As usual, we define powers of an operator $T$ inductively by $T^0 = \id$ with domain $\dom(T^0) = V$ and $T^{n+1}(v) = T (T^nv)$ with domain $\dom(T^{n+1}) = \{ v \in \dom(T^n): T(v)\in \dom(T)\}$. Moreover, we set $\dom(T^{\infty}) :=  \bigcap_{n\in\N_0}\dom(T^n)$.

\begin{definition}\label{ResSpec}
Let $T\in\closOP(V)$. For $s\in\F_0$ we define the operator
$$
\Q_s(T)x : = (T^2 - 2 {\rm Re}(s) T + |s|^2 \mathcal{I} )x, \qquad x \in \dom (T^2).
$$
The $S$-resolvent set of $T$ is the set
$$
\rho_S(T):=  \{  s\in\F_0: \Q_s(T)^{-1} \in\boundOP(V)\},
$$
and for $s\in\rho_S(T)$ the operator  $\Q_s(T)^{-1} : V\to \dom (T^2)$ is called the pseudo-resolvent of $T$ at $s$.
The $S$-spectrum of $T$ is defined as
$$
\sigma_S(T):=\F_0\setminus \rho_S(T).
$$
\end{definition}

\begin{lemma}
The $S$-spectrum of an operator $T\in\closOP(V)$ is axially symmetric and closed. If $T$ is bounded then $\sigma_S(T)$ is a nonempty, compact set contained in the closed ball $\overline{B_{\|T\|}(0)}$.
\end{lemma}
\begin{remark}
The following decomposition, analogue to the splitting in the classical theory, was introduced in \cite{Ghiloni:2013}:
\begin{enumerate}[(i)]
\item The {\em point $S$-spectrum} $\sigma_{Sp}(T)$ consists of all $s\in\F_0$ such that $\ker\Q_{s}(T)\neq \{0\}$.
\item The {\em continuous $S$-spectrum} $\sigma_{Sc}(T)$ consists of all $s\in\F_0$ such that $\ker\Q_{s}(T) = \{0\}$ and $\ran\Q_s(T)$ is dense in, but a proper subset of $V$.
\item The {\em residual $S$-spectrum} consists of all other points of the spectrum, i.e. of all $s\in\F_0$ such that $\ker\Q_s(T) = \{0\}$ but $\ran\Q_s(T)$ is not dense in $V$. 
\end{enumerate}
In the quaternionic setting, the point $S$-spectrum of $T$ equals the set of all right eigenvalues of~$T$ as it was shown in \cite{Colombo:2012}. The notion of eigenvalues must however be replaced by the notion of {\em eigenspheres} due to the axial symmetry of the $S$-spectrum. Indeed, if $v$ is an eigenvector of $T$ associated to the eigenvalue $s$, then $T(va) = T(v)a = vsa = (va)(a^{-1}sa)$ for $a\in\H$ and hence $va$ is an eigenvector associated to $a^{-1}sa$. Quaternions of the form $a^{-1}sa$ are exactly the elements of the sphere $[s]$ that is associated to $s$.
\end{remark}

\begin{definition}
Let $T\in\closOP(V)$. The left $S$-resolvent operator is defined as
\begin{equation}\label{SresolvoperatorL}
S_L^{-1}(s,T):= \Q_s(T)^{-1}\overline{s} -T\Q_s(T)^{-1}
\end{equation}
and the right $S$-resolvent operator is defined as
\begin{equation}\label{SresolvoperatorR}
S_R^{-1}(s,T):=-(T-\id \overline{s})\Q_s(T)^{-1}.
\end{equation}
\end{definition}
\begin{remark}\label{RkResExtension} Observe that  one obtains the right $S$-resolvent operator by formally replacing the variable $x$ in the right slice hyperholomorphic Cauchy kernel by the operator $T$. The same procedure yields
\begin{equation}\label{LResShort}
S_L^{-1}(s,T)v = -\Q_s(T)^{-1}(T-\overline{s}\id)v,\quad\text{for }v\in\dom(T)
\end{equation}
for the left $S$-resolvent operator. This operator is not defined on the entire space $V$, but only on  the domain $\dom(T)$ of $T$. Exploiting the fact that $ \Q_s(T)^{-1}$ and $T$ commute on $\dom(T)$, one can overcome this problem: commuting $T$ and $\Q_s(T)^{-1}$ in \eqref{LResShort} yields \eqref{SresolvoperatorL}.  The operator $\Q_s(T) = T^2 - 2\Re(s)T + |s|^2\id$ maps $\dom(T^2)$ to $V$. Hence, the pseudo-resolvent $\Q_s(T)^{-1}$ maps $V$ to $\dom(T^2)\subset \dom(T)$ if $s\in\rho_S(T)$. Since $T$ is closed and $\Q_s(T)^{-1}$ is bounded, equation \eqref{SresolvoperatorL} defines a continuous and therefore bounded right linear operator on the entire space $V$. Hence, the left resolvent $S_L^{-1}(s,T)$ is the natural extension of the operator \eqref{LResShort} to all of~$V$. In particular, if $T$ is bounded, then $S_L^{-1}(s,T)$ can directly be defined by~\eqref{LResShort}.

When considering left linear operators, one must obviously modify the definition
of the right $S$-resolvent operator for the same reasons.
\end{remark}

\begin{theorem}\label{ResProp}Let $T\in\closOP(V)$ and let $s\in\rho_S(T)$.
\begin{enumerate}[(i)]
\item  If $T$ and $s$ commute, then
\[\Q_s(T)^{-1} = (T-\overline{s}\id)^{-1}(T-s\id)^{-1}\]
 and the $S$-resolvent operators reduce to the classical resolvent, that is
\[S_L^{-1}(s,T) = S_R^{-1}(s,T) = (s\id - T)^{-1}.\]
\item The function $s\mapsto S_R^{-1}(s,T)$ is left slice hyperholomorphic and the function $s\mapsto S_L^{-1}(s,T)$ is right slice hyperholomorphic on $\rho_S(T)$.
\item The right $S$-resolvent operator satisfies the right $S$-resolvent equation
\begin{equation}\label{reseqR}
sS_R^{-1}(s,T)v-S_R^{-1}(s,T)Tv=\id v, \quad v\in \dom (T).
\end{equation}
The left $S$-resolvent operator satisfies the left $S$-resolvent equation
\begin{equation}\label{reseqL}
S_L^{-1}(s,T)s v - TS_L^{-1}(s,T)v = \id v\qquad v\in V.
\end{equation}
\end{enumerate}
\end{theorem}

Finally, the $S$-resolvent equation is the analogue of the classical resolvent equation. Note that it involves both $S$-resolvent operators and cannot be stated just for one of them. The $S$-resolvent equation has been proved in \cite{Alpay:2015} assuming that $T$ is bounded. In \cite{fracpow} it was generalized to the case of unbounded operators, in which one has to take possible problems  into account that concern the domains of definition of the operators.

\begin{theorem}[$S$-resolvent equation]Let $T\in\closOP(V)$. For  $s,p \in  \rho_S(T)$ with $s\notin[p]$, we have
\begin{equation}\label{resEQ}
\begin{split}
S_R^{-1}(s,T)S_L^{-1}(p,T)=&\big[[S_R^{-1}(s,T)-S_L^{-1}(p,T)]p
\\
&
-\overline{s}[S_R^{-1}(s,T)-S_L^{-1}(p,T)]\big](p^2-2s_0p+|s|^2)^{-1}.
\end{split}
\end{equation}
\end{theorem}

The definition of the $S$-resolvent operator allows us to define the $S$-functional calculus, which generalizes the Riesz-Dunford-functional calculus for holomorphic functions to our settings. In the following, we denote the domain of a function $f$ by $\fdom(f)$.
\begin{definition}\label{BSCalc}[$S$-functional calculus for bounded operators]
Let $T\in\boundOP(V)$. We define for any $f\in\rhol(\sigma_S(T))$
\begin{equation}\label{SCalcIntB}
f(T) := \frac{1}{2\pi} \int_{\partial(U\cap\C_I) }f(s)\,ds_I\, S_R^{-1}(s,T)
\end{equation}
and for $f\in\lhol(\sigma_S(T))$ 
\begin{equation*}
f(T) := \frac{1}{2\pi}\int_{\partial(U\cap\C_I)}S_L^{-1}(s,T)\,ds_I\,f(s),
\end{equation*}
where $I$ is an arbitrary imaginary unit and $U$ is any bounded slice Cauchy domain with $\sigma_S(T)\subset U$ and $\overline{U}\subset\fdom(f)$ that is also a slice domain.  These integrals are independent of the choice of the imaginary unit $I \in\ss$ and of the slice Cauchy domain $U$.
\end{definition}

 We say that a function $f$ is left slice hyperholomorphic at infinity if it is left slice hyperholomorphic on the set ${\{s\in\F_0: r<|s|\}}$ for some $r>0$ and the limit $\lim_{s\to\infty} f(s)$ exists. In this case we define
\[f(\infty) := \lim_{s\to\infty}f(s).\]
Similarly, we define right slice hyperholomorphicity at infinity.
\begin{definition}
Let $T\in\closOP(T)$. We denote the set of all functions $f\in\lhol(\sigma_S(T))$  that are left slice hyperholomorphic at infinity by $f\in\lhol(\sigma_S(T)\cup\{\infty\})$ and the set of all functions $f\in\rhol(\sigma_S(T))$  that are right slice hyperholomorphic at infinity by $f\in\rhol(\sigma_S(T)\cup\{\infty\})$.
\end{definition}

As in the complex case the $S$-functional calculus for unbounded operators is defined using a transformation of the unbounded operator into a bounded one.
For $\alpha\in\R$ we consider the function $\Phi_\alpha:\F_0\cup\{\infty\}\to\F_0\cup\{\infty\}$ defined by $\Phi_\alpha(s) = (s-\alpha)^{-1}$ for $s\in\F_0\setminus\{\alpha\}$, $\Phi_\alpha(\alpha) = \infty$ and $\Phi_\alpha(\infty) = 0$.
\begin{definition}{\rm
Let $T\in\closOP(V)$ be such that $\rho_S(T)\cap\R\neq\emptyset$, let $\alpha\in \rho_S(T)\cap\R$ and set $A = {(T-\alpha\id)^{-1} }= - S_L^{-1}(\alpha,T)$. The map $f  \mapsto f\circ\Phi_\alpha^{-1}$ defines a bijective relation between $\lhol(\sigma_S(T)\cup\{\infty\})$ and $\lhol(\sigma_S(A))$ resp. between $\rhol(\sigma_S(T)\cup\{\infty\})$ and $\rhol(\sigma_S(A))$. For any $f\in \lhol(\sigma_S(T)\cup\{\infty\})$ and any $f\in \rhol(\sigma_S(T)\cup\{\infty\})$ we define
\[ f(T) := (f\circ\Phi_{\alpha}^{-1})(A).\]}
\end{definition}
This definition is independent of the choice of $\alpha\in\rho_S(T)\cap\R$. Moreover, an integral representation corresponding to the one in \eqref{SCalcIntB} holds true as the next theorem shows.
\begin{theorem}\label{CSCalcInt}
Let $ T\in\closOP(V)$ with $\rho_S(T)\cap\R\neq\emptyset$. If $f\in\lhol(\sigma_S(T)\cup\{\infty\})$, then 
\[f(T) = f(\infty)\id + \frac1{2\pi}\int_{\partial(U\cap\C_I)} S_L^{-1}(s,T)\,ds_I\,f(s),\]
and if $f\in\rhol(\sigma_S(T)\cup\{\infty\})$, then
\[f(T) = f(\infty)\id + \frac1{2\pi}\int_{\partial(U\cap\C_I)} f(s)\,ds_I\,S_R^{-1}(s,T),\]
for any unbounded slice Cauchy domain $U$ with $\sigma_S(T)\subset U$ and $\overline{U}\subset \fdom(f)$ that is also a slice domain and any imaginary unit $I\in\ss$.
\end{theorem}

The functional calculi defined above are consistent with algebraic operations on the underlying function classes such as addition, multiplications with scalars from the left resp. right and multiplication and composition with intrinsic functions.

\section{A direct approach to the $S$-functional calculus for unbounded operators}
The technique of reducing the functional calculus for unbounded operators to the one of bounded operators is very useful in the classical complex setting. In the quaternionic or Clifford setting it has one disadvantage: it only applies to operators whose $S$-resolvent set contains a real point. Otherwise the map $s \mapsto (s-\lambda)^{-1}$ does not correspond to the $S$-resolvent operators at $\lambda$. In fact, it is then not even slice hyperholomorphic. The natural candidates to replace this function, the left and right Cauchy kernels,  are not intrinsic. Since the underlying concepts (spectral mapping, compatibility with composition of functions etc.) only apply to intrinsic functions and since a composed function is in general only slice hyperholomorphic, if the inner function is intrinsic, the left and right Cauchy kernels cannot be used to reduce the problem of defining a functional calculus for unbounded operators to the bounded case either. We therefore choose a direct approach, similar to the one \citeauthor{Taylor:1951} chose in \cite{Taylor:1951} for the complex setting, and define the $S$-functional calculus also for operators in $\closOP(V)$ by a Cauchy-integral.

The results in the following sections will often be stated for left and right slice hyperholomorphic functions. We will only give the proofs for the left slice hyperholomorphic case since the proofs of the other case are similar with obvious modifications.

\subsection{Some remarks on slice Cauchy domains}
The following theorem is well known in the complex case. Implicitly, it has also been assumed to hold true in our settings but, to the best of the author's knowledge, it has never been stated explicitly, which we shall do for the sake of completeness.

\begin{theorem}\label{CDExist}
Let $C$ be a closed and let $O$ be an open axially symmetric subset of $\F_0$ such that $C\subset O$ and such that $\partial O$ is nonempty and bounded. Then there exists a slice Cauchy domain $U$ such that $C\subset U$ and $\overline{U}\subset O$ and such that $U$ is unbounded if $O$ is unbounded. 
\end{theorem}
\begin{proof}
Let $I\in\ss$ and set $C_I = C\cap\C_I$ and $O_I= O\cap\C_I$. We cover the plane $\C_I$ by a honeycomb network of non-overlapping congruent hexagons of side $\delta/4$ with 
\[0<\delta < \dist (C_I, O_I^c) := \inf\{|z-z'|:z\in C_I, z'\in O_I^c\},\]
 where $ O_I^c$ denotes the complement of $O_I$ in $\C_I$ and we choose this network symmetric with respect to the real axis. We call the closure of such a hexagon a cell and denote the set of all cells in our network by $\mathfrak{S}$. Set  
 \[S := \bigcup\{\Delta\in \mathfrak{S}: \Delta\cap O_I^c \neq \emptyset\}.\]
 By standard arguments, we deduce that $U_I := S^c$ is a Cauchy domain in $\C_I$ such that $C_I\subset U_I$ and $\overline{U_I}\subset O_I$, which is unbounded if $O_I$ is unbounded. We refer to the proof of \cite[Theorem ~3.3]{Taylor:1951} for the technical details.  Since both the network of hexagons and the set $O_I^c$ are symmetric with respect  to the real axis, the set $S$ and in turn also $U_I$ are symmetric with respect to the real axis.

Now set $U := [U_I]$, where $[U_I]$ is the axially symmetric hull of $U_I$. Since $U_I$ is symmetric with respect to the real axis, we have $U_I=U\cap\C_I$. Moreover, as $C_I\subset U_I $ and $\overline{U_I}\subset O_I$, we obtain $C = [C_I]\subset [U_I] = U$ and $\overline{U} = [\overline{U_I}] \subset [O_I] = O$. If $O$ is unbounded then $O_I$  and $U_I$ are unbounded. Thus, $U$ is unbounded too.

It remains to show that $U$ is actually a slice Cauchy domain. Let $J\in\ss$ and observe that $U\cap\C_J = \{z_0 + Jz_1: z_0 + I z_1 \in U_I\}$ because $U$ is axially symmetric and $U_I = U\cap\C_I$. Since the mapping $\Phi: z_0 + Iz_1 \mapsto z_0 + J z_1$ is a homeomorphism from $\C_I$ to $\C_J$ and the set $U_I$ is a Cauchy domain in $\C_I$, we conclude that $U\cap\C_J = \Phi(U_I)$ is a Cauchy domain.
 
\end{proof}

The boundary of a slice Cauchy domain in a complex plane $\C_I$ is of course symmetric with respect to the real axis. Hence, it can be fully described by the part that lies in the upper half plane $\C_I^+:=\{x_0+Ix_1: x_0\in\R, x_1\geq 0\}$. We specify this idea in the following.

\begin{definition}
For any path $\gamma:[0,1]\to\C_I$, we define the paths $(-\gamma)(t):=\gamma(1-t)$ and $\overline{\gamma}(t): = \overline{\gamma(t)}$.
\end{definition}
\begin{lemma} \label{PathSplit}
Let $\gamma$ be a Jordan curve in $\C_I$ whose image is symmetric with respect to the real axis. Then $\gamma_+ := \gamma\cap\C_I^+$ consists of a single curve and $\gamma = \gamma_+ \cup \gamma_-$ with $\gamma_- := -\overline{\gamma_+}$.
\end{lemma}
\begin{proof}
Since its image is symmetric with respect to the real axis, $\gamma$ must take values in the upper and in the lower complex halfplane. Hence, as it is closed and continuous, it intersects the real line at least twice: once passing from the lower to the upper halfplane and once passing from the upper to the lower halfplane. Consider now a parametrization $\gamma(t)$, $t\in[0,1]$ of $\gamma$ with constant speed such that $\gamma(0)\in\R$ and such that $\gamma(t)\in\C_I^+$ for $t$ small enough. Then $\overline{\gamma}(t):=\overline{\gamma(t)}$ defines a parametrization of $\gamma$ with inverse orientation and constant speed because the image of $\gamma$ is symmetric with respect to the real axis. Since $(-\gamma)(t):=\gamma(1-t)$ is also a parametrization of $\gamma$ with inverse orientation and the same speed and starting point, we deduce $-\gamma = \overline{\gamma}$ and in turn $\gamma = - \overline{\gamma}$. Thus, $\gamma(1/2) = (-\overline{\gamma})(1/2) = \overline{\gamma(1/2)}$ and hence $\gamma(1/2)\in\R$. Moreover, there are no other points of $\gamma$ that lie on the real line: if $\gamma(\tau)\in\R$ for some $\tau\notin\{0,1/2\}$, then $\gamma(\tau) = \overline{\gamma(\tau)} = \gamma(1-\tau)$, which yields in a contradictions as $\gamma$ does not intersect itself.

Therefore, $\gamma_+(t):=\gamma(t/2)$, $t\in[0,1]$ takes values in $\C_I^+$. Otherwise, by continuity, it would have to intersect the real line when passing from the upper to the lower halfplane, which is impossible by the above argumentation. Moreover, the image of $\gamma_+$  coincides with $\gamma\cap\C_I^+$ because $\gamma = -\overline{\gamma}$ and hence
\[\gamma\setminus \gamma_+ = \left\{\gamma(t): \frac12 < t <1 \right\} = \left\{\overline{\gamma(t)}:0<t<\frac12\right\} = \left\{\overline{\gamma_+(t)}: 0<t<1\right\}, \]
which is a subset of $\C_I\setminus\C_I^+$ as $\gamma_+(t)\in\C_I^+$.

Finally, $\gamma(t) = \gamma_+(2t)$ if $t\in[0,1/2]$ and  $\gamma(t) = \overline{\gamma_+}(2-2t)$ if $t\in[1/2,1]$ and hence $\gamma = \gamma_+\cup\gamma_-$.

\end{proof}
Let now $U$ be a slice Cauchy domain and consider any $I\in\ss$. The boundary ${\partial(U\cap\C_I)}$ of $U$ in $\C_I$ consists of a finite union of piecewise continuously differentiable Jordan curves and is symmetric with respect to the real axis. Hence, whenever a curve $\gamma$ belongs to $\partial(U\cap\C_I)$, the curve $-\overline{\gamma}$ belongs to $\partial (U\cap\C_I)$ too. We can therefore decompose $\partial(U\cap\C_I)$ as follows:  
\begin{itemize}
\item First define $\gamma_{+,1},\ldots,\gamma_{+,\kappa}$ as those Jordan curves that belong to ${\partial(U\cap\C_I)}$ and lie entirely in the upper complex halfplane $\C_I^+$. Then the curves $ -\overline{\gamma_{+,1}}, \ldots, -\overline{\gamma_{+,\kappa}}$ are exactly those Jordan curves that belong to $\partial(U\cap\C_I)$ and lie entirely in the lower complex halfplane $\C_I^-$
\item In a second step consider the curves $\gamma_{\kappa+1},\ldots, \gamma_{N}$ that belong to $\partial(U\cap\C_I)$ and take values both in $\C_I^+$ and $\C_I^-$. Define $\gamma_{+,\ell}$ for $\ell = \kappa+1,\ldots,N$ as the part of $\gamma_{\ell}$ that lies in $\C_I^+$ and $\gamma_{-,\ell}$ as the part of $\gamma_{\ell}$ that lies in $\C_I^-$, cf. \Cref{PathSplit}.
\end{itemize}
Overall, we obtain the following decomposition of $\partial(U\cap\C_I)$:
\[ \partial(U\cap\C_I) = \bigcup_{1\leq\ell\leq N} \gamma_{+,\ell}\cup -\overline{\gamma_{+,\ell}}.\]
\begin{definition}\label{UpperPart}
We call the set $\{\gamma_{1,+},\ldots, \gamma_{N,+}\}$ the part of $\partial(U\cap\C_I)$ that lies in $\C_I^+$. 
\end{definition}

\subsection{A direct definition and algebraic properties of the $S$-functional calculus for closed operators}
In order to define the $S$-functional calculus for arbitrary operators in $\closOP(V)$ with nonempty $S$-resolvent set properly, we have to show that the respective Cauchy integral is independent of the choice the the slice Cauchy domain over which and the complex plane on which we integrate. We follow the strategy of \cite{Colombo:2011}.
\begin{theorem}\label{PreCalc}
Let $T\in\closOP(V)$ with $\rho_S(T)\neq\emptyset$. If $f\in\lhol(\sigma_S(T)\cup\{\infty\})$, then there exists an unbounded slice Cauchy domain $U$ such that $\sigma_S(T)\subset U$ and $\overline{U}\subset\fdom(f)$ and
\begin{equation}\label{HalfCalc}
\frac{1}{2\pi}\int_{\partial( U\cap\C_I)} S_L^{-1}(s,T) \, ds_I \, f(s) \in\boundOP(V),
\end{equation}
where the value of this integral is the same for any choice of the imaginary unit $I\in\ss$ and for any choice of $U$ satisfying the above conditions. 

Similarly, if $f\in\rhol(\sigma_S(T)\cup\{\infty\})$, then there exists an unbounded slice Cauchy domain $U$ such that $\sigma_S(T)\subset U$ and $\overline{U}\subset\fdom(f)$ and
\[\frac{1}{2\pi}\int_{\partial( U\cap\C_I)} f(s)\,ds_I\, S_R^{-1}(s,T)\in\boundOP(V),\]
where the value of this integral is the same for any choice of the imaginary unit $I\in\ss$ and for any choice of $U$ satisfying the above conditions.
\end{theorem}
\begin{proof}
Let $f\in\lhol(\sigma_S(T)\cup\{\infty\})$  and $p\in\rho_S(T)$. Since $\rho_S(T)$ is open, there exists a closed ball $\overline{B_{\varepsilon}(p)}\subset\rho_S(T)$ and since $\rho_S(T)$ is axially symmetric we have
\[ \left[\overline{B_{\varepsilon}(p)}\right] = \{ s = s_0 + I s_1\in V: (s_0 - p_0)^2 + (s_1 - p_1)^2 \leq \varepsilon\}\subset\rho_S(T).\]
The existence of the slice Cauchy domain $U$ now follows from \Cref{CDExist} applied with $C = \sigma_S(T)$ and $O = \fdom(F)\cap \left[\overline{B_{\varepsilon}(p)}\right]^c$.

 The integral defines a bounded operator because the boundary of $U$ in $\C_I$ consists of a finite set of closed piecewise differentiable Jordan curves and the integrand is continuous and hence bounded on the compact set $\partial(U\cap\C_I)$. 

We now show the independence of the slice Cauchy domain. Consider first the case of another unbounded slice Cauchy domain $U'$ such that $\sigma_S(T)\subset U'$ and $\overline{U'}\subset U$. Then $W = U\setminus {\overline{U'}}$ is a bounded slice Cauchy domain and 
\[\partial(W\cap\C_I) = \partial(U\cap\C_I) \cup - \partial (U'\cap\C_I),\] 
where $-\partial(U'\cap\C_I)$ denotes the inversely orientated boundary of $U'$ in $\C_I$. Moreover, the function $s\mapsto S_L^{-1}(s,T)$ is right and the function $s\mapsto f(s)$ is left slice hyperholomorphic on $\overline{W}$. Thus, \Cref{CauchyThm} implies
\begin{align*}
 0 &= \frac{1}{2\pi} \int_{\partial(W\cap\C_I)}S_L^{-1}(s,T)\, ds_I\, f(s) \\
 &= \frac{1}{2\pi}\int_{\partial(U\cap\C_I)} S_L^{-1}(s,T)\, ds_I\, f(s) - \frac{1}{2\pi}\int_{\partial(U'\cap\C_I)}S_L^{-1}(s,T)\, ds_I\, f(s).
 \end{align*}
If however $\overline{U'}$ is not contained in $U$, then $U\cap U'$ is an axially symmetric open set with nonempty and bounded boundary and it contains $\sigma_S(T)$. By \Cref{CDExist} there exist a third slice Cauchy domain $W$ such that $\sigma_S(T)\subset W$ and $\overline{W}\subset U\cap U'$ and by the above argumentation all of them yield the same operator in \eqref{HalfCalc}.

Finally, we consider another imaginary unit $J\in\ss$ and another slice Cauchy domain $W$ with $\sigma_S(T)\subset W$ and $\overline{W}\subset U$. By the above argumentation and \Cref{Cauchy}, we have
\begin{gather*}
\frac{1}{2\pi}\int_{\partial(U\cap\C_I)} S_L^{-1}(s,T)\,ds_I\, f(s) = \frac{1}{2\pi}\int_{\partial(W\cap\C_I)} S_L^{-1}(s,T)\, ds_I\, f(s) \\
= \frac{1}{(2\pi)^2}\int_{\partial(W\cap\C_I)} S_L^{-1}(s,T)\,ds_I\,\int_{\partial(U\cap\C_J)} S_L^{-1}(p,s)\,dp_J\,f(p)\\
=-\frac{1}{(2\pi)^2}\int_{\partial(U\cap\C_J)}\int_{\partial(W^c\cap\C_I)} S_L^{-1}(s,T)\,ds_I\, S_L^{-1}(p,s)\,dp_J\,f(p),
\end{gather*}
where Fubini's theorem allows us to exchange the order of integration in the last equation because we can integrate a bounded function over a finite domain.
Now observe that $S_L^{-1}(p,s) = - S_R^{-1}(s,p)$ by \Cref{SProps} and that the left $S$-resolvent $S_L^{-1}(s,T)$ is right slice hyperholomorphic in $s$ on $\overline{W^c}$. Since any $p\in\partial (U\cap\C_I)$ belongs to  $W^c$ by our choices of $U$ and $W$, we deduce from \Cref{Cauchy} that
\begin{gather*}
\frac{1}{2\pi}\int_{\partial(U\cap\C_I)} S_L^{-1}(s,T)\,ds_I\, f(s) =\\
= \int_{\partial(U\cap\C_J)}\frac{1}{(2\pi)^2}\int_{\partial(W^c\cap\C_I)} S_L^{-1}(s,T)\,ds_I\, S_R^{-1}(s,p)\,dp_J\,f(p)\\ 
= \int_{\partial(U\cap\C_J)} S_L^{-1}(p,T)\,dp_J\,f(p).
\end{gather*}

\end{proof}
\begin{definition}\label{DefiCalc}
Let $T\in\closOP(V)$ with $\rho_S(T)\neq\emptyset$. For $f\in\lhol(\sigma_S(T)\cup\{\infty\})$, we define
\begin{equation}\label{LCalc}
 f(T) := f(\infty)\id + \frac{1}{2\pi}\int_{\partial(U\cap\C_I)} S_L^{-1}(s,T)\,ds_I\, f(s),
 \end{equation}
and for $f\in\rhol(\sigma_S(T)\cup\{\infty\})$, we define
\begin{equation}\label{RCalc}
f(T) := f(\infty) \id + \frac{1}{2\pi}\int_{\partial(U\cap\C_I)} f(s)\,ds_I\, S_R^{-1}(s,T),
\end{equation}
where $I\in\ss$ and $U$ is any slice Cauchy domain as in \Cref{PreCalc}.
\end{definition}
\begin{remark}\label{OldConsistent}
Obviously, by \Cref{CSCalcInt}, our approach is consistent with the one used in \cite{Colombo:2011} if $\rho_S(T)\cap\R\neq\emptyset$. Moreover, it also includes the case of bounded operators: if $f\in\lhol(\sigma_S(T))$ for a bounded operator $T$, then, since the slice domain $U$ in \Cref{BSCalc} is bounded and we do not require connectedness of $\fdom(f)$ in \Cref{DefiCalc}, we can choose $r>0$ such that $\overline{U}$ is contained in the ball $B_{r}(0)$. We might then extend $f$ to a function in $\lhol(\sigma_{S}(T)\cup\{\infty\})$, for instance by setting  $f(s) = c$ with $c\in\F$ on $\F_0\setminus B_{r}(0)$, and use the unbounded slice Cauchy domain $(\F_0\setminus B_{r}(0))\cup U$ in \eqref{LCalc}. But since the left $S$-resolvent is then right slice hyperholomorphic on $\F_0\setminus B_{r}$ and $f(s)$ is left slice hyperholomorphic on this set, we obtain
\[ \begin{split}f(T) &= f(\infty)\id + \frac{1}{2\pi} \int_{-\partial( B_{r}(0)\cap\C_I)} f(s)\,ds_I\, S_L^{-1}(s,T) + \frac{1}{2\pi}\int_{\partial(U\cap\C_I)} f(s)\,ds_I\,S_L^{-1}(s,T)\\
=&\frac{1}{2\pi}\int_{\partial(U\cap\C_I)} f(s)\,ds_I\,S_L^{-1}(s,T)\end{split}\] 
because \Cref{CauchyThm} and \Cref{OPVal} imply that the sum of $f(\infty)\id$ and the integral over the boundary of $B_{r}(0)$ vanishes.
\end{remark}
\begin{example}\label{IDWorks}
Let $T\in\closOP(V)$ with $\rho_S(T)\neq\emptyset$. Consider the left slice hyperholomorphic function $f(s)= a$ for some $a\in\F$ and choose an arbitrary unbounded slice Cauchy domain $U$ with $\sigma_S(T)\subset U$ and an imaginary unit $I$. Then
\begin{equation}\label{ID}f(T) = f(\infty)\id + \frac{1}{2\pi}\int_{\partial(U\cap\C_I)} S_L^{-1}(s,T)\,ds_I\,f(s) = a\id, \end{equation}
because $f(\infty) = a$ and the integral vanishes by \Cref{CauchyThm} and \Cref{OPVal} as the left $S$-resolvent is right slice hyperholomorphic in $s$ on a superset of $\F_0\setminus U$ and vanishes at infinity. An analogue argument shows that also $f(T) = \id a=a\id$ if $f$ is considered right slice hyperholomorphic.
\end{example}

The following algebraic properties of the $S$-functional calculus immediately follow from the left and right linearity of the integral.
\begin{corollary}\label{AlgProp}
Let $T\in\closOP(V)$ with $\rho_S(T)\neq\emptyset$.
\begin{enumerate}[(i)]
\item If $f,g\in\lhol(\sigma_S(T)\cup\{\infty\})$ and $a\in\F$, then
\[(f+g)(T) = f(T) + g(T)\quad\text{and}\quad (fa)(T) = f(T)a.\]
\item If $f,g\in\rhol(\sigma_S(T)\cup\{\infty\})$ and $a\in\F$, then
\[(f+g)(T) = f(T) + g(T)\quad\text{and}\quad (af)(T) = af(T).\]
\end{enumerate}
\end{corollary}
\begin{remark}\Cref{PreCalc} ensures that these functional calculi are well-defined in the sense that they are independent of the choices of the imaginary unit $I\in\ss$ and the slice Cauchy domain $U$. However, they are not consistent unless one restricts to functions that are defined on axially symmetric slice domains. As we shall see in the following, there exist functions that are left and right slice hyperholomorphic such that \eqref{LCalc} and \eqref{RCalc} do not give the same operator, cf. \Cref{InconsRem} and \Cref{CEx}. However, at least for intrinsic functions  \eqref{LCalc} and \eqref{RCalc} are two representations for the same operator as we shall see now. 
\end{remark}
\begin{lemma}\label{BABABA}
Let $T\in\closOP(V)$ with $\rho_S(T)\neq\emptyset$ and let $f\in\intrin(\sigma_S(T)\cup\{\infty\})$. Furthermore consider a slice Cauchy domain $U$ such that $\sigma_S(T)\subset U$ and $\overline{U}\subset\fdom(f)$ and some imaginary unit $I\in\ss$. If $\gamma_{1},\ldots,\gamma_{N}$ is the part of $\partial(U\cap\C_I)$ that lies in $\C_I^+$ as in \Cref{UpperPart}, then
\begin{equation}\begin{split}\label{PosPlanInt}
&\int_{\partial(U\cap\C_I)}f(s)\,ds_I\,S_R^{-1}(s,T) \\
=& \sum_{\ell=1}^N\int_{0}^{1}2\Re\left(f(\gamma_{\ell}(t))(-I)\gamma_{\ell}'(t)\overline{\gamma_{\ell}(t)}
\right)\Q_{\gamma_{\ell}(t)}^{-1}(T)\,dt\\
&-\sum_{\ell=1}^N\int_{0}^{1}2\Re\Big(f(\gamma_{\ell}(t))(-I)\gamma_{\ell}'(t)\Big)T\Q_{\gamma_{\ell}(t)}^{-1}(T)\,dt.
\end{split}
\end{equation}
\end{lemma}
\begin{proof}
We have
\begin{align*}
&\int_{\partial(U\cap\C_I)}f(s)\,ds_I\,S_R^{-1}(s,T) \\
=&\sum_{\ell=1}^N\int_{\gamma_{\ell}} f(s)\,ds_I\,S_R^{-1}(s,T) + \sum_{\ell=1}^N\int_{-\overline{\gamma_{\ell}}}f(s)\,ds_I\,S_R^{-1}(s,T)\\
=&\sum_{\ell = 1}^N\int_{0}^{1}f(\gamma_\ell(t)) (-I) \gamma_\ell'(t)\left(\overline{\gamma_{\ell}(t)}-T\right)\Q_{\gamma_{\ell}(t)}^{-1}(T)\,dt\\
&+\sum_{\ell = 1}^N\int_{0}^{1}f\left(\overline{\gamma_\ell(1-t)}\right) I \overline{\gamma_\ell'(1-t)}(\gamma_{\ell}(1-t)-T)\Q_{\overline{\gamma_{\ell}(1-t)}}^{-1}(T)\,dt.
\end{align*}
Since $f(\overline{x}) = \overline{f(x)}$ as $f$ is intrinsic and $\Q_{\overline{s}}(T) = \Q_{s}(T)$ for $s\in\rho_S(T)$, we get after a change of variables in the integrals of the second sum
\begin{align}
&\notag \int_{\partial(U\cap\C_I)}f(s)\,ds_I\,S_R^{-1}(s,T) \\
=&\notag \sum_{\ell = 1}^N\int_{0}^{1}f(\gamma_\ell(t)) (-I) \gamma_\ell'(t)\left(\overline{\gamma_{\ell}(t)}-T\right)\Q_{\gamma_{\ell}(t)}^{-1}(T)\,dt\\
&+\notag \sum_{\ell = 1}^N\int_{0}^{1}\overline{f(\gamma_\ell(t))(-I)\gamma_\ell'(t)}(\gamma_{\ell}(t)-T)\Q_{\gamma(t)}^{-1}(T)\,dt\displaybreak[1]\\
\notag\begin{split}
=& \sum_{\ell=1}^N\int_{0}^{1}2\Re\left(f(\gamma_{\ell}(t))(-I)\gamma_{\ell}'(t)\overline{\gamma_{\ell}(t)}
\right)\Q_{\gamma_{\ell}(t)}^{-1}(T)\,dt\\
&-\sum_{\ell=1}^N\int_{0}^{1}2\Re\Big(f(\gamma_{\ell}(t))(-I)\gamma_{\ell}'(t)\Big)T\Q_{\gamma_{\ell}(t)}^{-1}(T)\,dt.
\end{split}
\end{align}
\end{proof}
%\begin{remark}
%Observe that the integral representation \eqref{PosPlanInt} does not involve any multiplication with nonreal quaternions from the right. We shall come back to this, when we consider the $S$-functional calculus on spaces that are not endowed with a left-multiplication.
%\end{remark}

\begin{theorem}\label{IntRep}
Let $T\in\closOP(V)$ with $\rho_S(T)\neq \emptyset$. If $f\in\intrin(\sigma_S(T)\cup\{\infty\})$, then 
\[\frac{1}{2\pi}\int_{\partial(U\cap\C_I)} S_L^{-1}(s,T)\,ds_I\,f(s) = \frac{1}{2\pi}\int_{\partial(U\cap\C_I)} f(s)\,ds_I\, S_R^{-1}(s,T)\]
for any $I\in\ss$ and any slice Cauchy domain as in \Cref{PreCalc}.
\end{theorem}
\begin{proof}
Fix $U$ and $I$, let $\gamma_{1},\ldots\gamma_{N}$ be the part of $\partial(U\cap\C_I)$ that lies in $\C_I^+$  and write the integral involving the right $S$-resolvent as an integral over these paths as in \eqref{PosPlanInt}.
Any operator commutes with real numbers and $f(\gamma_{\ell}(t))$, $\gamma_{\ell}'(t)$ and $\overline{\gamma_{\ell}(t)}$ commute mutually since they all belong to the same complex plane $\C_I$. Hence
\begin{align*}
& \int_{\partial(U\cap\C_I)}f(s)\,ds_I\,S_R^{-1}(s,T) \\
=& \sum_{\ell=1}^N\int_{0}^{1}\Q_{\gamma(t)}^{-1}(T)2\Re\left(\overline{\gamma_{\ell}(t)}\gamma_{\ell}'(t)(-I)f(\gamma_{\ell}(t))\right)\,dt\\
&-\sum_{\ell=1}^N\int_{0}^{1}T\Q_{\gamma_{\ell}(t)}^{-1}(T)2\Re\Big(\gamma_{\ell}'(t)(-I)f(\gamma_{\ell}(t))\Big)\,dt
\displaybreak[1]\\
=&\sum_{\ell=1}^N \int_{0}^1 \left(T\Q_{\gamma_{\ell}(t)}^{-1}(T) - \Q_{\gamma_{\ell}(t)}^{-1}(T)\overline{\gamma_{\ell}(t)}\right)\gamma_{\ell}'(t)(-I) f(\gamma_{\ell}(t)) \,dt\\
&+\sum_{\ell=1}^N \int_{0}^1 \left(T\Q_{\overline{\gamma_{\ell}(t)}}^{-1}(T) - \Q_{\overline{\gamma_{\ell}(t)}}^{-1}(T)\gamma_{\ell}(t)\right)\overline{\gamma_{\ell}'(t)}I \overline{f(\gamma_{\ell}(t))}\,dt
\displaybreak[1]\\
=& \sum_{\ell=1}^N \int_{\gamma_{\ell}}S_L^{-1}(s,T)\,ds_I\,f(s)+\sum_{\ell=1}^N\int_{-\overline{\gamma_{\ell}}}S_L^{-1}(s,T)\,ds_I\,f(s)\\
=& \int_{\partial(U\cap\C_I)} S_L^{-1}(s,T)\,ds_I\,f(s).
\end{align*}

\end{proof}

\begin{corollary}\label{IntWell}
Let $T\in\closOP(V)$ with $\rho_S(T)\neq \emptyset$. If $f\in\intrin(\sigma_S(T)\cup\{\infty\})$, then \eqref{LCalc} and \eqref{RCalc} give the same operator.
\end{corollary}

Recall that a function $f$ on $U$ is called locally constant if every point $x\in U$ has a neighborhood $B_x\subset U$ such that $f$ is constant on $U$. A locally constant function $f$ is constant on every connected subset of the its domain. Thus, since every sphere $[x]$ is connected, the function $f$ is constant on every sphere if its domain $U$ is axially symmetric, i.e. it is of the form $f(x) = c(x_0,x_1)$, where $c$ is locally constant on an appropriate subset of $\R^2$. Therefore $f$ can be considered a left and a right slice function and it is even left and right slice hyperholomorphic because the partial derivatives of a locally constant function vanish.

\begin{lemma}\label{BHolSplit}
A function $f$ is left and right slice hyperholomorphic if and only if 
\(f = c + \tilde{f}\), where $c$ is a locally constant slice function and $\tilde{f}$ is intrinsic.
\end{lemma}
\begin{proof}
Obviously any function that admits a decomposition of this type is both left and right slice hyperholomorphic. 

Assume on the other hand that $f$ is left and right slice hyperholomorphic such that $f(x) = \alpha(x_0,x_1) + I_x \beta(x_0,x_1)$ and $f(x) = \hat{\alpha}(x_0,x_1) +\hat{\beta}(x_0,x_1)I_x$. The compatibility condition \eqref{SymCond} implies
\[ \alpha(x_0,x_1) = \frac{1}{2}\left(f(x) + f(\overline{x})\right) = \hat{\alpha}(x_0,x_1),\]
from which we deduce $I\beta(x_0,x_1) = f(x_I) - \alpha(x_0,x_1) = \hat{\beta}(x_0,x_1)I$ for any $I\in\ss$, where $x_I = x_0 + Ix_1$. Hence we have
\[I\beta(x_0,x_1) I^{-1} = \hat{\beta}(x_0,x_1).\] If we choose $I = I_{\beta(x_0,x_1)}$, then $I$ and $\beta(x_0,x_1)$ commute and we obtain $\beta(x_0,x_1) = \hat{\beta}(x_0,x_1)$. Moreover, $\beta(x_0,x_1)$ commutes with every $I\in\ss$ because $I\beta(x_0,x_1) = \hat{\beta}(x_0,x_1)I = \beta(x_0,x_1) I$, which implies that $\beta(x_0,x_1)$ is real. 

Since $\beta$ takes real values, its partial derivatives $\frac{\partial}{\partial x_0}\beta(x_0,x_1)$ and $\frac{\partial}{\partial x_1}\beta(x_0,x_1)$  are real-valued too. Thus, since $\alpha$ and $\beta$ satisfy the Cauchy-Riemann-equations \eqref{CR}, the partial derivatives of $\alpha$ also take real-values. 

Now define $\tilde{\alpha}(x_0,x_1) = \Re(\alpha(x_0,x_1))$ and $\tilde{\beta}(x_0,x_1) = \beta(x_0,x_1)$ and set $\tilde{f} (x) = \tilde{\alpha}(x_0,x_1) + I_x\beta(x_0,x_1)$ and $c(x) = f(x) - \tilde{f}(x) = \Im(\alpha(x_0,x_1))$. Obviously, $\tilde{\alpha}$ and $\tilde{\beta}$ satisfy the compatibility condition \eqref{SymCond}. They also satisfy the Cauchy-Riemann-equations \eqref{CR} because $\alpha$ and $\beta$ do and
\[ \frac{\partial}{\partial x_i}\tilde{\alpha}(x_0,x_1) = \frac{\partial}{\partial x_i} \Re(\alpha(x_0,x_1)) = \Re\left(\frac{\partial}{\partial x_i} \alpha(x_0,x_1) \right) = \frac{\partial}{\partial x_i}\alpha(x_0,x_1),\]
Therefore $\tilde{f}$ is a left slice hyperholomorphic function with real-valued components, thus intrinsic.

It remains to show that $c$ is locally constant. Since $\tilde{c}(x)=\Im(\alpha(x_0,x_1))$, it depends only on $x_0$ and $x_1$ but not on the imaginary unit $I_x$ and is therefore constant on every sphere $[x]$ with $x\in U$. Moreover, as the sum of two slice hyperholomorphic functions, it is left (and right) slice hyperholomorphic and thus its restriction $c_I$ to any complex plane $\C_I$ is an $\F$-valued holomorphic function. But
\[ c_I'(x) = \frac{\partial}{\partial x_0} c_I(x) =   \frac{\partial}{\partial x_0} f(x) -  \frac{\partial}{\partial x_0}\tilde{f}(x) = 0\quad x\in U\cap\C_I\]
and hence $c$ is locally constant on $U\cap\C_I$. If $x\in U$, we can therefore find a neighborhood $B_{I_x}$ of $x$  in $U\cap\C_{I_x}$ such that $c_{I_x}$ is constant on $B_{I_x}$. Since $c$ is constant on any sphere, it is even constant on the axially symmetric hull $B = [B_{I_x}]$ of $B_{I_x}$, which is a neighborhood of $x$ in $U$.

\end{proof}

\begin{corollary}\label{ASJ}
Let $T\in\closOP(V)$ such that $T\in\rho_S(T)\neq\emptyset$ and let $f$ be both left and right slice hyperholomorphic on $\sigma_S(T)$ and at infinity. If $\fdom(f)$ is connected, then \eqref{LCalc} and \eqref{RCalc} give the same operator.
\end{corollary}
\begin{proof}
By applying \Cref{BHolSplit} we obtain a decomposition $f = c + \tilde{f}$ of $f$ into the sum of a locally constant function $c$ and an intrinsic function $\tilde{f}$. Since $\dom(f)$ is connected,  $c$ is even constant. Thus, \Cref{IntWell} and \Cref{IDWorks} imply
\begin{align*}
&f(\infty)\id +\frac{1}{2\pi} \int_{\partial(U\cap\C_I)} f(s)\,ds_I\, S_R^{-1}(s,T)\\
=&c\left(\id +\frac{1}{2\pi} \int_{\partial(U\cap\C_I)}ds_I\, S_R^{-1}(s,T)\right) + \tilde{f}(\infty)\id +\frac{1}{2\pi} \int_{\partial(U\cap\C_I)} \tilde{f}(s)\,ds_I\, S_R^{-1}(s,T) \\
=& c\id + \tilde{f}(T) = \id c + \tilde{f}(T) \\
=&\left(\id +\frac{1}{2\pi} \int_{\partial(U\cap\C_I)}S_L^{-1}(s,T)\,ds_I\right) + \tilde{f}(\infty)\id +\frac{1}{2\pi} \int_{\partial(U\cap\C_I)}  S_L^{-1}(s,T)\,ds_I\,\tilde{f}(s)\\
=& f(\infty)\id +\frac{1}{2\pi} \int_{\partial(U\cap\C_I)}  S_L^{-1}(s,T)\,ds_I\,f(s),
\end{align*}
where $U$ and $I$ are chosen as in \Cref{DefiCalc}.

\end{proof}
\begin{remark}\label{InconsRem}
We point out that \Cref{ASJ} does not hold true in general. The $S$-functional calculus has usually been considered for functions that are defined on connected sets, namely on axially symmetric slice domains. Hence, the calculi for left and right slice hyperholomorphic functions were consistent as we have seen in \Cref{ASJ}. 

However, this restriction occurred only due to the reasons explained in \Cref{History}. Since it excludes a wide class of functions, in particular those that generate spectral projections as they are studied in \Cref{ProjSec}, it is worthwhile to remove it. The price one has to pay in this case is that the two functional calculi become inconsistent. Indeed, in \Cref{ASJ} the function $c$ is constant since $\fdom(f)$ is connected and hence, by \Cref{IDWorks}, the functional calculi for left and right slice hyperholomorphic functions yield $c(T) = c\id$ and $c(T) = \id c$, respectively. Since the identity operator commutes with every constant $c$, these operators coincide.

If on the contrary $\fdom(f)$ is not connected, then $c$  is only locally constant, i.e. it will in general be of the form $\sum \chi_{\Delta_i}(s) c_i$, where the $\Delta_i$ are disjoint axially symmetric sets and $\chi_{\Delta_i}$ denotes the characteristic function of $\Delta_i$, which is obviously intrinsic. The functional calculi for left and right slice hyperholomorphic functions yield then $c(T) = \sum \chi_{\Delta_i}(T)c_i$ and $c(T) = \sum c_i\chi_{\Delta_i}(T)$, respectively.  These two operators coincide only if the operators $\chi_{\Delta_i}$ commute with the scalars $c_i$. As we will see in \Cref{ProjSec}, the operators $\chi_{\Delta_i}(T)$ are spectral projections onto invariant subspaces of the operator $T$. Since the operator $T$ is right linear, its invariant subspaces are right subspaces of $V$. But if a projection $\chi_{\Delta_i}(T)$ commutes with with any scalar, then $av = a\chi_{\Delta_i}(T)v = \chi_{\Delta}(T) av \in \chi_{\Delta_i}(T) V$ for any $v\in\chi_{\Delta_i}(T)V$ and any $a\in\F$ and thus $\chi_{\Delta_i}(T)V$ is also a left and therefore a two-sided subspace of $V$. In general, this is not true: the invariant subspaces obtained from spectral projections are only right sided. Hence, the projections $\chi_{\Delta_i}(T)$ do not necessarily commute with any scalar and it might be that  $\sum \chi_{\Delta_i}(T)c_i\neq \sum c_i\chi_{\Delta_i}(T)$, i.e. the two functional calculi give different operators for the same function.

An explicit example for this situation is given in \Cref{CEx}.
\end{remark}

\section{The product rule and polynomials in $T$}

In order to prove the product rule for our functional calculus, we recall Lemma~3.23 of \cite{Alpay:2015}. Observe that, for the reasons explained in \Cref{History}, the lemma was originally stated assuming that $U$ is a slice domain. However, the same proof works also in the case that $U$ is a bounded slice Cauchy domain.
\begin{lemma}\label{AlpayLem}
Let $B\in\boundOP(V)$, let $U$ be a bounded slice Cauchy domain and let $f\in\intrin(U)$. For $p\in U$ and any $I\in\ss$, we have
\[ Bf(p) = \frac{1}{2\pi}\int_{\partial(U\cap\C_I)} f(s)\,ds_I\,(\overline{s}B-Bp)(p^2-2s_0p+|s|^2)^{-1}.\]  
\end{lemma}

\begin{theorem}\label{ProdThm}
Let $T\in\closOP(V)$ with $\rho_S(T)\neq 0$. If $f\in\intrin(\sigma_S(T)\cup\{\infty\})$ and $g\in\lhol(\sigma_S(T)\cup\{\infty\})$, then
\begin{equation}\label{Prod}
(fg)(T) = f(T)g(T).
\end{equation}
Similarly, if  $f\in\rhol(\sigma_S(T)\cup\{\infty\})$ and $g\in\intrin(\sigma_S(T)\cup\{\infty\})$, then the product rule \eqref{Prod} also holds true.
\end{theorem}
\begin{proof}
By \Cref{PreCalc}, there exist unbounded slice Cauchy domains $U_p$ and $U_s$ such that $\sigma_S(T)\subset U_p$ and $\overline{U_p}\subset U_s$ and $\overline{U_s} \subset \fdom(f)\cap\fdom(g)$. The subscripts $s$ and $p$ indicate the respective variable of integration in the following computation. Moreover, we use the notation $[\partial O]_I := \partial(O\cap\C_I)$ for an axially symmetric set $O$ in order to obtain compacter formulas.

Recall that by \Cref{IntRep} the operator $f(T)$ can represented using the left and the right $S$-resolvent operator and hence
\begin{align*} f(T) g(T) = &\left(f(\infty)\id + \frac{1}{2\pi} \int_{[\partial U_s]_I} f(s)\,ds_I\,S_R^{-1}(s,T)\right)\cdot\\ 
&\qquad\qquad\quad\cdot\left(g(\infty)\id + \frac{1}{2\pi}\int_{[\partial U_p]_I}S_L^{-1}(p,T)\,dp_I\,g(p)\right).
 \end{align*}
 For the product of the integrals, the $S$-resolvent equation \eqref{resEQ} gives us that
\begin{align*}
&\int_{[\partial U_s]_I}f(s)\,ds_I\,S_R^{-1}(s,T)\int_{[\partial U_p]_I}S_L^{-1}(p,T)\,dp_I\,g(p)\\
=&\int_{[\partial U_s]_I}\int_{[\partial U_p]_I}f(s)\,ds_I\,S_R^{-1}(s,T)S_L^{-1}(p,T)\,dp_I\,g(p)\\
=&\int_{[\partial U_s]_I}\int_{[\partial U_p]_I}f(s)\,ds_I\,S_R^{-1}(s,T) p(p^2-2s_0p+|s|^2)^{-1}\,dp_I \,g(p)\\
&-\int_{[\partial U_s]_I}\int_{[\partial U_p]_I}f(s)\,ds_I\, S_L^{-1}(p,T)p(p^2-2s_0p+|s|^2)^{-1}\,dp_I \,g(p)\\
&-\int_{[\partial U_s]_I}\int_{[\partial U_p]_I}f(s)\,ds_I\,\overline{s}S_R^{-1}(s,T)(p^2-2s_0p+|s|^2)^{-1}\,dp_I \,g(p)\\
&+\int_{[\partial U_s]_I}\int_{[\partial U_p]_I}f(s)\,ds_I\,\overline{s}S_L^{-1}(p,T) (p^2-2s_0p+|s|^2)^{-1}\,dp_I\, g(p).
\end{align*}
For the sake of readability, let us denote these last four integrals by $I_1,\ldots I_4$. 

If $r>0$ is large enough, then $\F_0\setminus U_s$ is entirely contained in $B_r(0)$. In particular, $W := B_r(0)\cap U_p$ is then a bounded slice Cauchy domain with $\partial(W\cap\C_I) = \partial(U_p\cap\C_I)\cup \partial(B_r(0)\cap\C_I)$. From \Cref{AlpayLem}, we deduce
\begin{align*}
 I_1 =& \int_{[\partial U_s]_I}f(s)\,ds_I\,S_R^{-1}(s,T)\int_{[\partial U_p]_I} p(p^2-2s_0p+|s|^2)^{-1}\,dp_I \,g(p)\\
 =& \int_{[\partial U_s]_I}f(s)\,ds_I\,S_R^{-1}(s,T)\int_{[\partial W]_I} p(p^2-2s_0p+|s|^2)^{-1}\,dp_I\, g(p)\\
&- \int_{[\partial U_s]_I}f(s)\,ds_I\,S_R^{-1}(s,T)\int_{[\partial B_r(0)]_I} p(p^2-2s_0p+|s|^2)^{-1}\,dp_I\, g(p)\\
=&- \int_{[\partial U_s]_I}f(s)\,ds_I\,S_R^{-1}(s,T)\int_{[\partial B_r(0)]_I} p(p^2-2s_0p+|s|^2)^{-1}\,dp_I\, g(p),
\end{align*}
where the last equality follows from Cauchy's integral theorem since the function $p\mapsto p(p^2-2s_0p+|s|^2)^{-1}$ is left slice hyperholomorphic and the function $p\mapsto g(p)$ is right slice hyperholomorphic on $\overline{W}$ by our choice of $U_s$ and $U_p$. 
If we let $r$ tend to $+\infty$ and apply Lebesgue's theorem in order to exchange limit and integration, the inner integral tends to $2\pi g(\infty)$ and hence
\[I_1 = - 2\pi\left(\int_{[\partial U_s]_I}f(s)\,ds_I\,S_R^{-1}(s,T)\right) g(\infty).\]
We also have
\begin{align*}
-I_2+I_4 =& \int_{[\partial U_s]_I}\int_{[\partial U_p]_I} f(s)\,ds_I\,\left(\overline{s}S_L^{-1}(p,T)-pS_L^{-1}(p,T)\right)\cdot\\
&\phantom{\int_{[\partial U_s]_I}\int_{[\partial U_p]_If(s)\,ds_I\,\qquad\qquad\ \,} }\cdot(p^2-2s_0p+|s|^2)^{-1}\,dp_I\,g(p).
\end{align*}
and applying Fubini's theorem allows us to change the order of integration. If we now set $W = B_{r}(0)\cap U_s$ with $r$ sufficiently large we obtain as before a bounded slice Cauchy domain with $\partial (W\cap\C_I) = \partial(U_s\cap\C_I)\cup\partial (B_r(0)\cap\C_I)$. From \Cref{AlpayLem}, applied with $B= S_L^{-1}(p,T)$, we deduce  
\begin{align*}
-I_2+I_4 
=& \int_{[\partial U_p]_I}\int_{[\partial W]_I} f(s)\,ds_I\,\left(\overline{s}S_L^{-1}(p,T)-pS_L^{-1}(p,T)\right)\cdot\\
&\phantom{\int_{[\partial U_p]_I}\int_{[\partial U_s]_If(s)\,ds_I\,\qquad\qquad\ \,} }\cdot(p^2-2s_0p+|s|^2)^{-1}\,dp_I\,g(p)\\
&- \int_{[\partial U_p]_I}\int_{[\partial B_r(0)]_I} f(s)\,ds_I\,\left(\overline{s}S_L^{-1}(p,T)-pS_L^{-1}(p,T)\right)\cdot\\
&\phantom{\int_{[\partial U_p]_I}\int_{[\partial U_s]_If(s)\,ds_I\,\qquad\qquad\ \,} }\cdot(p^2-2s_0p+|s|^2)^{-1}\,dp_I\,g(p)\\
=& 2\pi\int_{[\partial U_p]_I}S_L^{-1}(p,T)f(p)\,dp_I\,g(p)\\
&- \int_{[\partial U_p]_I}\int_{[\partial B_r(0)]_I} f(s)\,ds_I\,\overline{s}S_L^{-1}(p,T)(p^2-2s_0p+|s|^2)^{-1}\,dp_I\,g(p)\\
&- \int_{[\partial U_p]_I}\int_{[\partial B_r(0)]_I} f(s)\,ds_I\,pS_L^{-1}(p,T)(p^2-2s_0p+|s|^2)^{-1}\,dp_I\,g(p)\\
\end{align*}
Observe that the third integral tends to zero as $r\to\infty$. For the second, we obtain by applying Lebesgue's theorem
\begin{gather*}
\int_{[\partial U_p]_I}\int_{[\partial B_r(0)]_I} f(s)\,ds_I\,\overline{s}S_L^{-1}(p,T)(p^2-2s_0p+|s|^2)^{-1}\,dp_I\,g(p)\\
=\int_{[\partial U_p]_I}\left(\int_{0}^{2\pi} f(re^{i\phi})r^2S_L^{-1}(p,T)(p^2-2r\cos(\phi)p+r^2)^{-1}\,d\phi\right)\,dp_I\,g(p)\\
\overset{r\to+\infty}{\longrightarrow}2\pi f(\infty)\int_{[\partial U_p]_I}S_L^{-1}(p,T)\,dp_I\,g(p).
\end{gather*} 
Since $f$ is intrinsic, $f(p)$ commutes with $dp_I$, and hence
\begin{align*}
-I_2+I_4 =& 2\pi \int_{[\partial U_p]_I}S_L^{-1}(p,T)\,dp_I\,f(p)g(p) \\
&- 2\pi f(\infty)\int_{[\partial U_p]_I}S_L^{-1}(p,T)\,dp_I\,g(p).
\end{align*}
Finally, we consider the integral $I_3$. If we set again $W = B_r(0)\cap U_p$ with $r$ sufficiently large, then
\begin{align*}
-I_3 = &-\int_{[\partial U_s]_I}\int_{[\partial W]_I}f(s)\,ds_I\,\overline{s}S_R^{-1}(s,T)(p^2-2s_0p+|s|^2)^{-1}\,dp_I \,g(p)\\
 &+\int_{[\partial U_s]_I}\int_{[\partial B_r(0)]_I}f(s)\,ds_I\,\overline{s}S_R^{-1}(s,T)(p^2-2s_0p+|s|^2)^{-1}\,dp_I \,g(p).
\end{align*} 
By our choice of $U_s$ and $U_p$, the functions and $p\mapsto (p^2-2s_0p + |s|^2)^{-1}$ and $p\mapsto g(p)$ are left resp. right slice hyperholomorphic on $\overline{V}$. Hence, Cauchy's integral theorem implies that the first integral equals zero. Letting $r$ tend to infinity, we can apply Lebesgue's theorem in order to exchange limit and integration and we see that
\begin{align*}
-I_3 &= \int_{[\partial U_s]_I}\int_{[\partial B_r(0)]_I}f(s)\,ds_I\,\overline{s}S_R^{-1}(s,T)(p^2-2s_0p+|s|^2)^{-1}\,dp_I \,g(p)\to 0.
\end{align*} 
Altogether, we obtain
\begin{align*}
&\frac{1}{(2\pi)^2}\int_{[\partial U_s]_I}f(s)\,ds_I\,S_R^{-1}(s,T)\int_{[\partial U_p]_I}S_L^{-1}(p,T)\,dp_I\,g(p)\\
 =&- \frac{1}{2\pi}\left(\int_{[\partial U_s]_I}f(s)\,ds_I\,S_R^{-1}(s,T)\right) g(\infty)+ \frac{1}{2\pi}\int_{[\partial U_p]_I}S_L^{-1}(p,T)\,dp_I\,f(p)g(p)\\
  &-  f(\infty)\frac{1}{2\pi}\int_{[\partial U_p]_I}S_L^{-1}(p,T)\,dp_I\,g(p).
\end{align*}
We thus have
\begin{align*}
f(T)g(T) =& f(\infty)g(\infty)\id +  f(\infty)\frac{1}{2\pi}\int_{[\partial U_p]_I}S_L^{-1}(p,T)\,dp_I\,g(p)\\
&+\left(\frac{1}{2\pi}\int_{[\partial U_s]_I}f(s)\,ds_I\,S_R^{-1}(s,T) \right)g(\infty) \\
&+ \frac{1}{(2\pi)^2}\int_{[\partial U_s]_I}f(s)\,ds_I\,S_R^{-1}(s,T)\int_{[\partial U_p]_I}S_L^{-1}(p,T)\,dp_I\,g(p)\\
=& f(\infty)g(\infty)\id + \frac{1}{2\pi}\int_{[\partial U_p]_I}S_L^{-1}(p,T)\,dp_I\,f(p)g(p) = (fg)(T).
\end{align*}
\end{proof}

If the operator $T$ is bounded, then slice hyperholomorphic polynomials of $T$ belong to the class of functions that are admissible for the $S$-functional calculus. In the unbounded cases, this is not true, but the $S$-functional calculus is in some sense still compatible, at least with intrinsic polynomials. For such polynomial $P(s) = \sum_{k=0}^n a_ks^k$ with $a_k\in\R$, the operator $P(T)$ is as usual defined as the operator\[P(T)v := \sum_{k=0}^na_kT^kv\qquad v\in\dom(T^n).\]
\begin{lemma}\label{Commute}
Let $T\in\closOP(V)$ with $\rho_S(T)\neq\emptyset$, let $f\in\intrin(\sigma_S(T)\cup\{\infty\})$ and let $P$ be an intrinsic polynomial of degree $n\in\N_0$. If $v\in\dom(T^n)$, then $f(T)v\in\dom(T^n)$ and $f(T)P(T)v = P(T)f(T)v$.
\end{lemma}
\begin{proof}
We consider first the special case $p(s) = s$. Let $U$ be a slice Cauchy domain with $\sigma_S(T)\subset U$ and let $\{\gamma_{1},\ldots,\gamma_{n}\}$ be the part of $\partial(U\cap\C_I)$ in $\C_I^+$ for some $I\in\ss$, cf. \Cref{UpperPart}. We apply \Cref{BABABA} and write
\begin{align*}
&\int_{\partial(U\cap\C_I)}f(s)\,ds_I\,S_R^{-1}(s,T) \\
=& \sum_{\ell=1}^N\int_{0}^{1}2\Re\left(f(\gamma_{\ell}(t))(-I)\gamma_{\ell}'(t)\overline{\gamma_{\ell}(t)}
\right)\Q_{\gamma_{\ell}(t)}^{-1}(T)\,dt\\
&-\sum_{\ell=1}^N\int_{0}^{1}2\Re\Big(f(\gamma_{\ell}(t))(-I)\gamma_{\ell}'(t)\Big)T\Q_{\gamma_{\ell}(t)}^{-1}(T)\,dt.
\end{align*}
Observe that $\Q_{\gamma_i(t)}^{-1}(T)Tv = T\Q_{\gamma_i(t)}^{-1}(T)v$ for  $v\in\dom(T)$ and that $T$ also commutes with real numbers. By applying Hille's theorem for the Bochner integral, we can move $T$ in front of the integral and find
\begin{align*}
&\frac{1}{2\pi}\int_{\partial(U\cap\C_I)} f(s)\,ds_I\,S_R^{-1}(s,T) Tv\\
=&\sum_{i=1}^nT\frac{1}{2\pi}\int_{0}^12\Re\left(f(\gamma_i(t))(-I)\gamma_i'(t)\overline{\gamma_i(t)}\right)\Q_{\gamma_i(t)}^{-1}(T)v\\
&-\sum_{i=1}^nT\frac{1}{2\pi}\int_{0}^12\Re\Big(f(\gamma_i(t))(-I)\gamma_i'(t)\Big)T\Q_{\gamma_i(t)}^{-1}(T)v\\
&= T\frac{1}{2\pi}\int_{\partial(U\cap\C_I)} f(s)\,ds_I\,S_R^{-1}(s,T) v,
\end{align*}
where the last equation follows again from \eqref{PosPlanInt}. Finally, observe that $f(\infty) = \lim_{s\to\infty}f(s)$ is real since $f(s)\in\R$ for any $s\in\R$. Hence,
\begin{align*}
 f(T)Tv& = f(\infty)Tv + \frac{1}{2\pi}\int_{\partial(U\cap\C_I)}f(s)\,ds_I\,S_R^{-1}(s,T)Tv \\
 &= Tf(\infty)v + T\frac{1}{2\pi}\int_{\partial(U\cap\C_I)}f(s)\,ds_I\,S_R^{-1}(s,T)v = Tf(T)v.
 \end{align*}
In  particular, this implies $f(T)v\in\dom(T)$.

We show the general statement by induction with respect to the degree $n$ of the polynomial. If $n = 0$ then the statement follows immediately from \Cref{IDWorks}. Now assume that it is true for $n-1$ and consider $P(s) = a_ks^n + P_{n-1}(s)$, where $a_n\in\R$ and $P_{n-1}(s)$ is an intrinsic polynomial of degree lower or equal to $n-1$. For $v\in\dom(T^n)$ the above argumentation implies then $f(T)T^{n-1}v\in\dom(T)$ and
\begin{align*}
f(T)P(T)v &= f(T)a_nT^nv + f(T)P_{n-1}(T)v \\&= a_nTf(T)T^{n-1}v + f(T)P_{n-1}(T)v.
\end{align*}
From the induction hypothesis, we further deduce $f(T)T^{n-1}v = T^{n-1}f(T)v$ and $f(T)P_{n-1}(T)v = P_{n-1}(T)f(T)v$ and hence
\[f(T)P(T)v = a_nT^nf(T)v + P_{n-1}(T)f(T)v = P(T)f(T)v.\]
In particular, we see $f(T)v\in\dom(T^{n})$.

\end{proof}
As in the complex case, we say that $f$ has a zero of order $n$ at $\infty$ if the first $n-1$-coefficients in the Taylor series expansion of $s\mapsto f(s^{-1})$ at $0$ vanish and the $n$-th coefficient does not. Equivalently, $f$ has a zero of order $n$ if $\lim_{s\to\infty} f(s)s^n$ is bounded and nonzero. We say that $f$ has a zero of infinite order, if it vanishes on a neighborhood of $\infty$.
\begin{lemma}\label{ProdPoly}
Let $T\in\closOP(V)$ with $\rho_S(T)\neq\emptyset$ and assume that $f\in\intrin(\sigma_{S}(T)\cup\{\infty\})$ has a zero of order $n\in\N_0\cup\{+\infty\}$ at infinity.
\begin{enumerate}[(i)]
\item\label{LKA1} For any intrinsic polynomial $P$ of degree lower than or equal to $n$, we have $P(T)f(T) = (Pf)(T)$.
\item\label{LKA2}  If $v\in\dom(T^m)$ for some $m\in\N_0\cup\{\infty\}$, then $f(T)v\in\dom(T^{m+n})$.
\end{enumerate}
\end{lemma}
\begin{proof}
Assume first that $f$ has a zero of order greater than or equal to one at infinity and consider $P(s)=s$. Then $Pf\in\intrin(\sigma_{S}(T)\cup\{\infty\})$ and for $v\in V$
\[(Pf)(T)v =  \lim_{s\to\infty}sf(s) v + \frac{1}{2\pi}\int_{\partial(U\cap\C_I)}S_L^{-1}(s,T)\,ds_I sf(s)v,\]
with an appropriate slice Cauchy domain $U$ and any imaginary unit $I\in\ss$.
Since $s$ and $ds_I$ commute, we deduce from the left $S$-resolvent equation \eqref{reseqL} that 
\begin{align*}
&\frac{1}{2\pi}\int_{\partial(U\cap\C_I)}S_L^{-1}(s,T)\,ds_I\,sf(s)v \\
=&\frac{1}{2\pi}\int_{\partial(U\cap\C_I)}TS_L^{-1}(s,T)\,ds_I\, f(s)v + \frac{1}{2\pi}\int_{\partial(U\cap\C_I)}\,ds_I\,f(s)v
\end{align*}
Any sufficiently large Ball $B_{r}(0)$ contains $\partial U $. The function $f(s)v$ is then right slice hyperholomorphic on $\overline{B_r(0)\cap U}$ and Cauchy's integral theorem implies
\begin{gather*}
\frac{1}{2\pi}\int_{\partial(U\cap\C_I)}\,ds_I\,f(s)v = \lim_{r\to+\infty}- \frac{1}{2\pi}\int_{\partial(B_{r}(0)\cap\C_I)}\,ds_I\,f(s)v\\
=\lim_{r\to+\infty} -\frac{1}{2\pi}\int_{0}^{2\pi}re^{I\varphi} f(re^{I\varphi})v\,d\varphi = - \lim_{r\to+\infty}rf(r)v.
\end{gather*}
Thus, after applying Hille's theorem for the Bochner integral in order to write the operator $T$ in front of the integral, we obtain
\begin{align*} (Pf)(T)v 
&= T\frac{1}{2\pi}\int_{\partial(\cap\C_I)}S_L^{-1}(s,T)\,ds_I\,f(s)v = P(T)f(T)v.
\end{align*}
In particular, we see that $f(T)v\in\dom(T)$. 

We show \cref{LKA1} for monomials by induction and assume that it is true for $p(s) = s^{n-1}$ if $f$ has a zero of order greater than or equal to $n-1$ at infinity. If the order of $f$ at infinity is even greater than or equal to $n$, then $g(s) =  s^{n-1}f(s)$ has a zero of order at least 1 at infinity and, from the above argumentation and the induction hypothesis, we conclude for $P(s)=s^n$
\begin{align*}
(Pf)(T)v = Tg(T)v = TT^{n-1}f(T)v = T^nf(T)v,
\end{align*}
which implies also $f(T)v\in\dom(T^n)$. For arbitrary intrinsic polynomials the statement finally follows from the linearity of the $S$-functional calculus.
 
In order to show \cref{LKA2} assume first $v\in\dom(T^m)$ for $m\in\N$. If $f$ has a zero of order $n\in\N$ at infinity, then \cref{LKA1} with $P(s)= s^n$ and \Cref{Commute} imply
\[ (Pf)(T)T^mv =T^nf(T)T^mv = T^nT^mf(T)v = T^{m+n}f(T)v \]
and hence $f(T)v\in\dom(T^{m+n})$. Finally, if $m=+\infty$ then $v\in\dom(T^k)$ and hence $f(T)v\in\dom(T^{k+n})$ for any $k\in\N$. Thus, $v\in\dom(T^{\infty})$.

\end{proof}

\begin{corollary} \label{PolyClos}
Let $T\in\closOP(X)$ with $\rho_S(T)\neq\emptyset$. For any intrinsic polynomial $P$, the operator $P(T)$ is closed.
\end{corollary}
\begin{proof}
We choose $s\in\rho_S(T)$ and $n\in\N$ such that $m\leq 2n$, where $m$ is the degree of $P$. Then $f(p) = P(p)\Q_{s}(p)^{-n}$ belongs to $\intrin(\sigma_{S}(T)\cup\{\infty\})$ and has a zero of order $2n-m$ at infinity. Applying \Cref{ProdPoly}, we see that 
\[P(T)v  = P(T)\Q_{s}(T)^n\Q_{s}(T)^{-n}v = \Q_{s}(T)^nP(T)\Q_{s}(T)^{-n}v = \Q_{s}(T)^nf(T)v\]
for $v\in\dom(T^m)$. 
Since its inverse is bounded, the operator $\Q_{s}(T)^n$ is closed and in turn $P(T)$ is closed as it is the composition of a closed and a bounded operator. 

\end{proof}

\begin{corollary}\label{ClosInv}
Let $T\in\closOP(V)$ with $\rho_S(T)\neq\emptyset$. If $f\in\intrin(\sigma_S(T)\cup\{\infty\})$  has no zeros on $\sigma_S(T)$ and a zero of even order $n$ at infinity, then $\ran(f(T)) = \dom(T^n)$  and $f(T)$ is invertible in the sense of closed operators. If $\rho_S(T)\cap\R\neq\emptyset$, this holds true for any order $n\in\N$.
\end{corollary}
\begin{proof}
Let $P\in\rho_S(T)$ and set $k=n/2$. The function $h(s) = f(s)\Q_p(s)^k$ with $\Q_p(s) = s^2 - 2\Re(p)s+|p|^2$ belongs to $\intrin(\sigma_S(T)\cup\{\infty\})$ and does not have any zeros in $\sigma_S(T)$. Furthermore, $h(\infty) = \lim_{s\to\infty}h(s)$ is finite and nonzero. Hence, $s\mapsto h(s)^{-1}$  belongs to $\intrin(\sigma_S(T)\cup\{\infty\})$ and we deduce from \Cref{ProdThm}  that $h(T)$ is invertible in $\boundOP(V)$. \Cref{ProdThm} moreover implies $f(T) = \Q_p(T)^{-k}h(T)$. Now observe that $h(T)$ maps $V$ bijectively onto $V$ and that $\Q_p(T) ^{-k}$ maps $V$ onto $\dom(T^{2k}) = \dom(T^{n})$. Thus $\ran(f(T)) = \dom(T^n)$. 

Finally, $f(T)^{-1} := h^{-1}(T)\Q_p(T)^k$ is a closed operator because $h$ is bijective and continuous and $\Q_p(T)^k$ is closed by \Cref{PolyClos} and it satisfies $f(T)^{-1}f(T)v = v$ for $v\in V$ and $f(T)f(T)^{-1}v = v$ for $v\in\dom(T^n)$. Thus it is the inverse of $f(T)$.

In the case there exists a point  $a\in\rho_S(T)\cap\R$, a similar argumentation holds with $P(s) = (s-a)^n$ instead of $\Q_p(s)^k$. In particular, this allows us to include functions with a zero of odd order at infinity too.

\end{proof}

\section{The spectral mapping theorem and composite functions}
\begin{definition}
Let $T\in\closOP(V)$. We define the extended spectrum $\sigma_{SX}(T)$ as $\sigma_{S}(T)$ if $T$ is bounded and as $\sigma_S(T)\cup\{\infty\}$ otherwise. The extended resolvent set $\rho_{SX}(T)$ is the complement of $\sigma_{SX}(T)$ in $\F_0\cup\{\infty\}$.
\end{definition}

\begin{theorem}[Spectral Mapping Theorem] \label{SpecMap}
Let $T\in\closOP(V)$ with $\rho_S(T)\neq\emptyset$. If $f\in\intrin(\sigma_S(T)\cup\{\infty\})$, then
\(\sigma_S(f(T)) = f(\sigma_{SX}(T))\).
\end{theorem}
\begin{proof}
Let us first show the relation $\sigma_S(f(T)) \supset f(\sigma_{SX}(T))$. For $p\in\sigma_S(T)$ consider the function
\[g(s) := (f(s)^2 - 2\Re(f(p))f(s) - |f(p)|^2)(s^2 - 2\Re(p)s + |p|^2)^{-1},\]
which is defined on $\fdom(f)\setminus[p]$. If we set $p_{I_s} = p_0 + I_s p_1$, then $p_{I_s}$ and $s$ commute. Since $f$ is intrinsic, it maps $\C_I$ into $\C_I$ and hence $f(p_{I_s})$ and $f(s)$ commute too. Thus
\[g(s) = \frac{(f(s)-f(p_{I_s}))(f(s) - \overline{f(p_{I_s})})}{(s-p_{I_s})(s-\overline{p_{I_s}})}\]
and we can extend $g$ to all of $\fdom(f)$ by setting
\begin{equation}
\arraycolsep=1.4pt\def\arraystretch{2}
g(s) = 
\left\{\begin{array}{cc}
 \sderiv f(s) \left(\underline{f(p)}\,\underline{p}^{-1}\right)& s\in[p]\text{ if } p\notin\R\\
(\sderiv f(s))^2 & s = p, \text{ if }p \in \R 
\end{array}\right.,
\end{equation}
where $\underline{p} = \frac12(p - \overline{p})$ denotes the vectorial part of $p$. Now observe that 
\[(s^2 - 2\Re(p)s+|p|^2)g(s) = f(s)^2 + 2\Re(f(p))f(s) + |f(p)|^2\]
and that $g$ has zero of order greater or equal to 2 at infinity. Hence, we can apply the $S$-functional calculus to deduce from \Cref{ProdPoly}, \Cref{ProdThm} and \Cref{IDWorks} that
\[ (T^2 - 2\Re(p)T + |p|^2\id)g(T)v = (f(T)^2 + 2\Re(f(p))f(T) + |f(p)|\id)v \]
for any $v\in V$ and
\[ g(T)(T^2 - 2\Re(p)T + |p|^2\id g(T))v = (f(T)^2 + 2\Re(f(p))f(T) + |f(p)|\id)v\]
for $v\in\dom(T^2)$.
If $f(p) \in \rho_S(T)$, then 
\[\Q_{f(p)}(f(T)) = f(T)^2 - 2\Re(f(p))f(T)+|f(p)|\id\]
is invertible and $\Q_{f(p)}(f(T))^{-1}g(T) = g(T)\Q_{f(p)}(f(T))^{-1}$ is the inverse of $\Q_{p}(T) = T^2 - 2\Re(p)T+|p|^2\id$. Hence, $f(p) \notin \sigma_S(f(T))$ implies $p\notin\sigma_S(T)$ and as a consequence $p\in\sigma_S(T)$ implies $f(p)\in\sigma_S(T)$, that is $f(\sigma_S(T))\subset \sigma_S(f(T))$. 

Finally, observe that $f(\infty) =\lim_{p\to\infty}f(p)$ is real because $f$ is intrinsic and thus takes real values on the real line. If $T$ is unbounded and $f(\infty)\neq f(p)$ for any $p\in\sigma_S(T)$ (otherwise we already have $f(\infty)\in f(\sigma_S(T))\subset \sigma_S(f(T))$), then the function $h(s) = (f(s)-f(\infty))^2$ belongs to $\intrin(\sigma_S(T)\cup\{\infty\})$ and has a zero of even order $n$ at infinity but no zero in $\sigma_S(T)$. By \Cref{ClosInv}, the range of $h(T) = \Q_{f(\infty)}(f(T))$ is $\dom(T^n)$. Thus, it does not admit a bounded inverse and we obtain $f(\infty)\in\sigma_S(f(T))$. Altogether, we have $f(\sigma_{SX}(T))\subset\sigma(f(T))$.

In order to show the relation $\sigma_{S}(f(T))\subset f(\sigma_{SX}(T))$, we first consider a point $c  \in \sigma_S(f(T))$ such that $c  \neq f(\infty)$. We want to show $c  \in f(\sigma_{S}(T))$ and assume the converse, i.e. $f(s) - c $ has no zeros on $\sigma_S(T)$. 

If $c $ is real, then the function $h(s) = f(s) - c $ is intrinsic, has no zeros on $\sigma_S(T)$ and $\lim_{s\to\infty}h(s) = f(\infty) - c  \neq 0$. Hence, $h^{-1}(s) = (f(s)-c )^{-1}$ belongs to $\intrin(\sigma_S(T)\cup\{\infty\})$. Applying the $S$-functional calculus, we deduce from \Cref{ProdThm} that $h^{-1}(T)$ is the inverse of $f(T) - c \id$ and hence $\Q_c (f(T))^{-1} = (h^{-1}(T))^2$, c.f. \Cref{ResProp}, which is a contradiction as $c \in \sigma_S(f(T))$. Thus, $c  = f(p)$ for some $p\in\sigma_S(T)$.

If on the other hand $c $ is not real, then $f-c _I \neq 0$ for any $c _I = c_0 + Ic_1\in[c ]$. Indeed, $f(p) = \alpha(p_0,p_1) + I_p \beta(p_0,p_1) = c_0 + Ic_1$ would imply $I_p=I$ and $\alpha(p_0,p_1) = c_0$ and $\beta(p_0,p_1)=c_1$ as $\alpha$ and $\beta$ are real-valued because $f$ is intrinsic. This would in turn imply $f(p_{I_c}) = \alpha(p_0,p_1) + I_c\beta(p_0,p_1) = c$, which would contradict our assumption. Therefore, the function $h(s) = (f(s)^2 - 2\Re(c)f(s)+|c|^2)=(f(s)-c_{I_s})(f(s)-\overline{c_{I_s}})$  has no zeros on $\sigma_S(T)$. Moreover, since $f(\infty)$ is real, we have
\[h(\infty) = (f(\infty)-c)\overline{(f(\infty)-c)} = |f(\infty)-c|^2\neq 0\]
and hence $h^{-1}(s) = (f(s)^2 - 2\Re(c)f(s) + |c|^2)^{-1}$ belongs to $\intrin(\sigma_S(T)\cup\{\infty\})$. Applying the $S$-functional calculus, we deduce again from \Cref{ProdThm} that $h^{-1}(T)$ is the inverse of $\Q_c(T)$, which contradicts $c\in\sigma_S(f(T))$. Hence, there must exist some $p\in\sigma_S(T)$ such that $c=f(p)$. 

Altogether, we obtain $\sigma_S(f(T))\setminus\{f(\infty)\}$ is contained in $f(\sigma_{S}(T))$.

Finally, let us consider the case that the point $c = f(\infty)$ belongs to $\sigma_S(f(T))$. If $T$ is unbounded, then $\infty\in\sigma_{SX}(T)$ and hence $c\in f(\sigma_{SX}(T))$. If on the other hand $T$ is bounded, then there exists a function $g\in\intrin(\sigma_S(T)\cup\{\infty\})$ that coincides on an axially symmetric neighborhood $\sigma_S(T)$ with $f$ but satisfies $c\neq g(\infty)$. In this case $f(T) = g(T)$, as pointed out in \Cref{OldConsistent}, and we can apply the above argumentation with $g$ instead of $f$ to see that $c\in g(\sigma_S(T))= f(\sigma_S(T))$.

\end{proof}

\begin{theorem}
If $T\in\closOP(V)$ with $\sigma_S(T)\neq\emptyset$, then $P(\sigma_S(T)) = \sigma_S(P(T))$ for any intrinsic polynomial $P$.
\end{theorem}
\begin{proof}
The arguments are similar to those in the proof of \Cref{SpecMap}: in order to show that $P(\sigma_S(T))\subset\sigma_S(P(T))$, we consider the polynomial $\Q_{P(p)}(P(s)) = P(s)^2 - 2\Re(P(p))P(s) + |P(p)|^2$ for any $p\in\sigma_S(T)$. As $p$ and $\overline{p}$ are both zeros of $\Q_{P(p)}(P(s))$ (resp. as $p$ is a zero of even order of $\Q_{P(p)}(P(s)) = (P(s) - P(p))^2$ if $p$ is real), there exists an intrinsic polynomial $R(s)$ such that $\Q_{P(p)}(P(s)) = \Q_{p}(s)R(s)$. If $P(p)\notin\sigma_S(P(T))$, then $\Q_{P(p)}(P(T))$ is invertible and  \Cref{ProdPoly} and \Cref{IDWorks} imply that $\Q_{P(p)}(P(T))^{-1}R(T)$ is the inverse of $\Q_{p}(T)$, which is a contradiction because we assumed $p\in\sigma_S(T)$. Therefore $P(p)\in\sigma_S(P(T))$.

Conversely assume that $p\notin P(\sigma_S(T))$. Then the function $\Q_{p}(P(s)) = {P(s)^2-2\Re(p)P(s)+|p|^2}$ does not take any zero on $\sigma_S(T)$ and we conclude from \Cref{ClosInv} that $\Q_{p}(P(T))$ has a bounded inverse. Thus $p\notin \sigma_S(P(T))$ and in turn $\sigma_S(P(T))\subset P(\sigma_S(T))$.

\end{proof}

\begin{theorem}
Let $T\in\closOP(V)$ with $\rho_S(T)\neq\emptyset$. If $f\in\intrin(\sigma_S(T)\cup\{\infty\})$ and $g\in\lhol(f(\sigma_{SX}(T))$ or $g\in\rhol(f(\sigma_{SX}(T))$, then
\[(g\circ f)(T) = g(f(T)).\] 
\end{theorem}
\begin{proof}
Because of \Cref{OldConsistent}, we can assume w.l.o.g. that $f(\infty)\in f(\sigma_{SX}(T))$.
We apply \Cref{CDExist} in order to choose a slice Cauchy domain $U_p$ such that $\sigma_S(f(T)) =f( \sigma_{SX}(T) ) \subset U_p$ and $\overline{U_p}\subset\fdom(g)$ and a second slice Cauchy domain $U_s$ such that $\sigma_S(T)\subset U_s$ and  $\overline{U_s}\subset f^{-1}(U_p)\cap\fdom(f)$. The subscripts are chosen in order to indicate the respective variable of integration in the following computation. 

After choosing an imaginary unit $I\in\ss$, we deduce from \Cref{Cauchy}, Cauchy's integral formula, that 
\begin{align*}
&(g\circ f)(T) - (g\circ f)(\infty)\id \\
=& \frac{1}{2\pi}\int_{\partial(U_s\cap\C_I)} S_L^{-1}(s,T)\,ds_I\,(g\circ f)(s)\\
=& \frac{1}{2\pi}\int_{\partial(U_s\cap\C_I)}S_L^{-1}(s,T)\,ds_I\, \left(\frac{1}{2\pi}\int_{\partial( U_p\cap\C_I)}S_L^{-1}(p,f(s))\,dp_I\,g(p) \right).
\end{align*} 
Changing the order of integration by applying Fubini's theorem, we obtain
\begin{align*}
&(g\circ f)(T) - (g\circ f)(\infty)\id\\
=&\frac{1}{2\pi}\int_{\partial( U_p\cap\C_I)} \left(\frac{1}{2\pi}\int_{\partial(U_s\cap\C_I)}S_L^{-1}(s,T)\,ds_I\, S_L^{-1}(p,f(s))\right)dp_I\,g(p) \\
=& \frac{1}{2\pi}\int_{\partial( U_p\cap\C_I)}S_L^{-1}(p,f(T))\, dp_I\, g(p) \\
&- \frac{1}{2\pi}\int_{\partial( U_p\cap\C_I)}S_L^{-1}(p,f(\infty))\,dp_I\,g(p)\id\\
=& g(f(T)) - g(f(\infty))\id 
\end{align*}
and hence $ (g\circ f)(T) = g(f(T))$.

\end{proof}

\section{Spectral sets and projections}\label{ProjSec}

\begin{definition}
A subset $\sigma$ of $\sigma_{SX}(T)$ is called a spectral set if it is open and closed in $\sigma_{SX}(T)$. 
\end{definition}
Just as $\sigma_S(T)$ and $\sigma_{SX}(T)$, every spectral set is axially symmetric: if $s\in\sigma$ then the entire sphere $[s]$ is contained in $\sigma$. Indeed, the set $\sigma\cap[s]$ is then a nonempty, open and closed subset of $\sigma_{SX}(T)\cap[s]=[s]$. Since $[s]$ is connected this implies $\sigma\cap[s] = [s]$. Moreover, if $\sigma$ is a spectral set, then $\sigma' = \sigma_{SX}(T)\setminus\sigma$ is a spectral set too.

If $\sigma$ is a spectral set of $T$, then $\sigma$ and $\sigma'$ can be separated in $\F_0\cup\{\infty\}$ by axially symmetric open sets and hence \Cref{CDExist} implies the existence of two slice Cauchy domains $U_{\sigma}$ and $U_{\sigma}'$ containing $\sigma$ and $\sigma'$ respectively such that one of them is unbounded and $\overline{U}\cap\overline{U_{\sigma'}} = \emptyset$. We define
\[\chi_{\sigma}(x) :=\begin{cases} 1&\text{if } x\in U_{\sigma}\\ 0 & \text{if }x\in U_{\sigma}'.\end{cases} \]
The function $\chi_{\sigma}(x)$ obviously belongs to $\intrin(\sigma_S(T)\cup\{\infty\})$. 
\begin{definition}
Let $T\in\closOP(V)$ with $\rho_S(T)\neq\emptyset$ and let $\sigma\subset\sigma_S(T)$ be a spectral set of $T$. The spectral projection associated with $\sigma$ is the operator $E_\sigma :=\chi_{\sigma}(T)$ obtained by applying  the $S$-functional calculus to the function~$\chi_{\sigma}$. Furthermore, we define $V_{\sigma} := E_{\sigma} V$ and $T_{\sigma} = T|_{\dom(T_{\sigma})}$ with $\dom(T_{\sigma}) = \dom(T)\cap V_{\sigma}$.
\end{definition}
Explicit formulas for the operator $E_{\sigma}$ are
\[ E_{\sigma} =  \frac{1}{2\pi} \int_{\partial(U_{\sigma}\cap\C_I)} S_L^{-1}(s,T) \,ds_I=  \frac{1}{2\pi} \int_{\partial(U_{\sigma}\cap\C_I)} ds_I\,S_R^{-1}(s,T) \]
if $\sigma$ is bounded and
\[ E_{\sigma} = \id +  \frac{1}{2\pi} \int_{\partial(U_{\sigma}\cap\C_I)} S_L^{-1}(s,T) \,ds_I= \id +   \frac{1}{2\pi} \int_{\partial(U_{\sigma}\cap\C_I)} ds_I\,S_R^{-1}(s,T) \]
if $\sigma$ is unbounded, where the imaginary unit $I\in\ss$ can be chosen arbitrarilyarbitrarily.
\begin{corollary} Let $T\in\closOP(V)$ such that $\rho_S(T)\neq\emptyset$ and let $\sigma$ be a spectral set of $T$.
\begin{enumerate}[(i)]
\item The operator $E_{\sigma}$ is actually a projection, i.e. $E_{\sigma}^2 = E_{\sigma}$.
\item Set $\sigma' = \sigma_{SX}(T)\setminus \sigma$. Then $E_{\sigma} + E_{\sigma'} = \id$ and $E_{\sigma}E_{\sigma'}= E_{\sigma'}E_{\sigma} = 0$.
\end{enumerate}
\end{corollary}
\begin{proof}
This follows  immediately from the algebraic properties of the $S$-functional calculus shown in \Cref{AlgProp} and \Cref{ProdThm} as $\chi_{\sigma}^2 = \chi_{\sigma}$ and $\chi_{\sigma} + \chi_{\sigma'} = 1$ and $\chi_{\sigma}\chi_{\sigma'}=\chi_{\sigma'}\chi_{\sigma}=0$. 

\end{proof}

A right $\F$-module $X$ is the direct sum of two right submodules $X_1$ and $X_2$ of $X$, if every $v\in X$ can be written in a unique way as $v = v_1 + v_2$ with $v_i\in X_i$. In this case we write $X = X_1\oplus X_2$ and we call $X_1$ and $X_2$ complementary submodules. Unlike for vector spaces over a field, it is not true for modules that any submodules has a complement. (Though it can easily be verified that this hold true at least for the quaternionic setting.) A complement exists if and only if there exists a projection such that the respective submodule is either its kernel or its image \cite[Chapter~II~\S 1.9, Proposition~14]{Bourbaki:1998}. With this definition, the following lemma is immediate and the second one follows easily from  the continuity of projections on a Banach module, whose range is closed. 
\begin{lemma}\label{SuSp}
Let $A$, $B$, $M$  and $N$ be right linear submodules of a right $\F$-module $X$ such that $A\subset M$ and $B\subset M$. If $A\oplus B = M\oplus N$, then $A=M$ and $B=N$.
\end{lemma}
\begin{lemma}\label{DSuSp}
Let $A$, $B$, $M$ and $N$ be right linear subspaces of $V$ such that $A\subset M$, $B\subset N$ and such that $M$, $N$ and $M\oplus N$  are closed. Then $A\oplus B$ is dense in $M\oplus N$ if and only if $A$ is dense in $M$ and $B$ is dense in $N$.
\end{lemma}

\begin{theorem}\label{ProjGeneral}
Let $T\in\closOP(V)$ with $\rho_S(T)\neq\emptyset$ and let $E_{1}, E_{2}$ be projections on $V$ such that $E_1 + E_2 = I$ (and hence $E_1E_2 = E_2E_1  =0$).
Denote $V_i := E_i(V)$ and $\dom( T_i): = E_i(\dom(T))$ and assume that $T(\dom(T_i))\subset V_i$ such that $T_i:= T|_{\dom(T_i)}$ is a closed operator on the right Banach module $V_i$ . Then
\begin{enumerate}[(i)]
\item $E_iT v = TE_iv$ for $v\in\dom(T)$,
\item $\dom(T_i^2) = E_i(\dom(T^2))$ for $i\in\{1,2\}$,
\item  $\ran(\Q_s(T) )  = \ran(\Q_s(T_1))\oplus\ran(\Q_s(T_2))$ for any $s\in\F_0$,
\item \label{PGsig}$\sigma_{S}(T) = \sigma_S(T_1) \cup \sigma_S(T_2)$ and
\item \label{jj1}$\sigma_{Sp}(T) = \sigma_{Sp}(T_1) \cup \sigma_{Sp}(T_2)$.
\end{enumerate}
If moreover $\sigma_S(T_1)\cap\sigma_S(T_2) = \emptyset$, then
\begin{enumerate}[(i), resume]
\item \label{jj2}$\sigma_{Sc}(T) = \sigma_{Sc}(T_1)\cup \sigma_{Sc}(T_2)$ and
\item \label{jj3}$\sigma_{Sr}(T) = \sigma_{Sr}(T_1)\cup\sigma_{Sr}(T_2)$. 
\end{enumerate}
\end{theorem}
\begin{proof}
The assertions (i), (ii) and (iii) are obvious. Now assume that $s\in\rho_S(T)$. Then $\ran(\Q_{s}(T)) = V$ and from (iii) we deduce
\[ V_1\oplus V_2 = V = \ran(\Q_{s}(T)) = \ran(\Q_{s}(T_1)) \oplus\ran(\Q_{s}(T_2)).\]
As $\ran(\Q_{s}(T_i)) \subset V_i$, \Cref{SuSp} implies $\ran(\Q_{s}(T_i)) = V_i$  and hence $\Q_{s}(T_i)^{-1} = \Q_{s}(T)^{-1}|_{V_i}$ as $\Q_{s}(T_i) = \Q_{s}(T)|_{\dom(T_i^2)}$. Indeed, we have
\[ \Q_{s}(T)^{-1}\Q_{s}(T_i)v = \Q_{s}(T)^{-1}\Q_{s}(T)v = v \qquad\text{for }v\in\dom(T_i^2) \]
and, since $\Q_{s}(T)^{-1}v \in \dom(T_i^2)$ for $v\in\ V_i$, also 
\[\Q_{s}(T_i) \Q_{s}(T)^{-1}v = \Q_{s}(T)\Q_{s}(T)^{-1}v = v \qquad\text{for }v\in V_i. \]
Thus, $s\in\rho_S(T_1)\cap\rho_S(T_2)$. Conversely, if $s\in\rho_S(T_1)\cap\rho_S(T_2)$, then $\Q_{s}(T_1)^{-1}E_1 + \Q_{s}(T_2)^{-1}E_2$ is the inverse of $\Q_s(T)$ and hence $s\in\rho_S(T)$. Altogether, $\rho_S(T) = \rho_S(T_1)\cap\rho_S(T_2)$, which is equivalent to $\sigma_S(T) = \sigma_S(T_1)\cup\sigma_S(T_2)$. 

Obviously, $\sigma_{Sp}(T_i)\subset \sigma_{Sp}(T)$ as any $S$-eigenvector of $T_i$ is also an $S$-eigenvector of $T$ associated with the same  eigensphere. Conversely, if $v\neq 0$ is an $S$-eigenvector of $T$ associated with the eigensphere $[s] = s_0 + \ss s_1$, then set $v_i = E_iv$ and observe that
\[ 0 = \Q_{s}(T)v = \Q_{s}(T_1)v_1 + \Q_{s}(T_2)v_2.\]
As $\Q_{s}(T_i)v_i \in V_i$ and $V_1\cap V_2 = \{0\}$, this implies $\Q_{s}(T_i)v_i = 0$ for $i =\{1,2\}$. As $v\neq 0$, at least one of the vectors $v_i$ is nonzero and therefore an $S$-eigenvalue of $T_i$ associated to the eigensphere $[s]$. Thus $[s]\subset\sigma_{Sp}(T_1)\cup\sigma_{Sp}(T_2)$ and in turn $\sigma_{Sp}(T) = \sigma_{Sp}(T_1) \cup \sigma_{Sp}(T_2)$. 

We assume now that $\sigma_S(T_1)\cap\sigma_S(T_2) = \emptyset$. Then assertions (iv) and (v) imply that $s\in\sigma_{Sc}(T)\cup\sigma_{Sr}(T)$ if and only if $s\in\sigma_{Sc}(T_i)\cup\sigma_{Sr}(T_i)$ for either $i=1$ or $i=2$. We assume w.l.o.g. $s\in\sigma_{Sc}(T_1)\cup\sigma_{Sr}(T_1)$ and thus $s\in\rho_S(T_2)$. As $\ran(\Q_{s}(T_2)) = V_2$, we deduce from (iii) that and \Cref{DSuSp} that $\ran(\Q_{s}(T))$ is dense in $V = V_1\oplus V_2$ if and only if $\ran(\Q_{s}(T_1))$ is dense in $V$. In other words: $s\in\sigma_{Sc}(T)$ if and only if $s\in \sigma_{Sc}(T_1)$ and in turn $s\in\sigma_{Sr}(T)$ if and only if $s\in\sigma_{Sr}(T_1)$.

\end{proof}

\begin{theorem}
Let $T\in\sigma_{S}(T)$ with $\rho_S(T)\neq\emptyset$ and let $\sigma\subset\sigma_{S}(T)$ be a spectral set of $T$. Then
\begin{enumerate}[(i)]
\item $E_{\sigma}(\dom(T))\subset\dom(T)$
\item $T(\dom(T)\cap V_{\sigma})\subset V_{\sigma}$
\item \label{sfds}$\sigma = \sigma_{SX}(T_{\sigma})$
\item \label{asd1}$\sigma \cap \sigma_{Sp}(T) = \sigma_{Sp}(T_{\sigma})$
\item \label{asd2}$\sigma\cap\sigma_{Sc}(T) = \sigma_{Sc}(T_{\sigma})$
\item \label{asd3} $\sigma\cap\sigma_{Sr}(T) = \sigma_{Sr}(T_{\sigma})$
\end{enumerate}
If the spectral set $\sigma$ is bounded, then we further have
\begin{enumerate}[(i), resume]
\item \label{XInf}$V_{\sigma}\subset \dom(T^{\infty})$
\item \label{ttse}$T_{\sigma}$ is a bounded operator on $V_{\sigma}$. 
\end{enumerate}
\end{theorem}
\begin{proof}
Assertion (i) follows from the definition of $E_{\sigma}$ and \Cref{Commute}. In order to prove (ii), we observe that if $v\in\dom(T)\cap V_{\sigma}$, then $E_{\sigma}v = v$. Hence, we deduce from \Cref{Commute} that  $E_{\sigma}Tv = TE_{\sigma}v =  Tv$, which implies $Tv\in V_{\sigma}$.

If $\sigma$ is bounded, then we can choose $U_{\sigma}$ bounded and hence $\chi_{\sigma}$ has a zero of infinite order at infinity. We conclude from \Cref{ProdPoly} that $v = E_{\sigma}v = \chi_{\sigma}(T)v \in\dom(T^{\infty})$ for any $v\in V_{\sigma}$ and hence \ref{XInf} holds true. In particular, $V_{\sigma}\subset \dom(T)$. Therefore $T_{\sigma}$ is a bounded operator on $V_{\sigma}$ as it is closed and everywhere defined.

We show now assertion (iii) and consider first a point $s\in\F_0\setminus\sigma$. We show $s\in\rho_S(T_{\sigma})$. For an appropriately chosen slice Cauchy domain $U_{\sigma}$, the function $f(s) := \Q_{s}(p)^{-1}\chi_{U_{\sigma}}(s)$ belongs to $\intrin(\sigma_{S}(T)\cup \{\infty\})$. By \Cref{ProdPoly} and \Cref{Commute}, we have
\[ f(T) \Q_{s}(T)v = \chi_{U_{\sigma}}(T)v = E_{\sigma}v, \qquad\text{for } v\in \dom(T^2)\cap V_{\sigma}\]
and 
\[ \Q_{s}(T)f(T) v = \chi_{U_{\sigma}}(T)v = E_{\sigma}v = v \qquad\text{for } v \in V_{\sigma}.\]
Hence, $\Q_{s}(T_\sigma) = \Q_{s}(T)|_{V_{\sigma}\cap\dom(T^2)}$ has the inverse $f(T)|_{V_{\sigma}}\in\boundOP(X_{\sigma})$. Thus $s\in\rho_S(T_{\sigma})$ and in turn $\sigma_{S}(T_{\sigma})\subset \sigma \cap\F_0 =:\sigma_{1}$. The same argumentation applied to $T_{\sigma'}$ with $\sigma' = \sigma_{SX}(T)\setminus \sigma$ shows that $\sigma_{S}(T_{\sigma'})\subset \sigma'\cap\F_0:= \sigma_2$. But by \ref{PGsig} in \Cref{ProjGeneral}, we have
\[ \sigma_{S}(T_{\sigma})\cup\sigma_{S}(T_{\sigma}) = \sigma_{S}(T) = \sigma_1 \cup \sigma_2\]
and hence $\sigma_{S}(T_{\sigma}) = \sigma_1 = \sigma\cap\F_0$  and $\sigma_{S}(T_{\sigma'}) = \sigma_2 = \sigma'\cap\F_0$. If $\sigma$ is bounded, then this is equivalent to \ref{sfds} because of \ref{ttse}. If $\sigma$ is not bounded, then $\infty\in\sigma$ and $T$ is not bounded on $X$. However, in this case $\sigma'$ is bounded and hence $T_{\sigma}\in\boundOP(V_{\sigma})$. But as $T = T_{\sigma}E_{\sigma} + T_{\sigma'}E_{\sigma'}$, we conclude that $T_{\sigma}$ is unbounded as $T$ is unbounded. Hence $\infty \in \sigma_{SX}(T_{\sigma})$ and \ref{ttse} holds true also in this case.

Finally, \cref{asd1,asd2,asd3} are direct consequences of \cref{jj1,jj2,jj3} in \Cref{ProjGeneral} as we know now that $\sigma_{S}(T_{\sigma})$ and $\sigma_{S}(T_{\sigma'})$ are disjoint.

\end{proof}

\begin{example}\label{CEx}
Choose a generating basis $I$, $J$ and $K = IJ$ of $\H$ and consider the quaternionic right-linear operator $T$ on $V =\H^2$ that is defined by its action on the two right linearly independent right eigenvectors $v_1 = (1,I)^T$ and $v_2 = (J,-K)^T$, namely
\begin{equation*}
\begin{pmatrix} 1\\ I\end{pmatrix}\mapsto \begin{pmatrix}0\\ 0\end{pmatrix} \qquad \text{and}\qquad \begin{pmatrix} J\\ -K\end{pmatrix}\mapsto \begin{pmatrix}-K\\-J\end{pmatrix}= \begin{pmatrix}J\\-K\end{pmatrix}I.
\end{equation*}
Its matrix representation is
\begin{equation*}
T  = \frac12 \begin{pmatrix}-I & 1 \\ -1 & -I\end{pmatrix}.
\end{equation*}
Since,  for operators on finite-dimensional spaces, the $S$-spectrum coincides with the set of right-eigenvalues, cf. \cite{Colombo:2012}, we have $\sigma_S(T) = \sigma_R(T) = \{0\}\cup\ss$. Indeed, we have
\begin{align*}
\Q_s(T) &= \frac{1}{2}\begin{pmatrix} -1 & -I \\ I & -1\end{pmatrix} - s_0\begin{pmatrix}- I & 1 \\ -1 & -I \end{pmatrix} + |s|^2 \begin{pmatrix}1 & 0 \\ 0 & 1\end{pmatrix}\\
&= \begin{pmatrix} -\frac 12 + |s|^2 +s_0I & -s_0 -\frac 12 I \\ s_0 +\frac 12 I & -\frac12 + |s|^2 + s_0  I \end{pmatrix}
\end{align*}
and hence 
\[\Q_s(T)^{-1} = |s|^{-2}(-1 +  2 Is_0 + |s|^2)^{-1} \begin{pmatrix} -\frac12 + |s|^2 + Is_0 & \frac{1}{2}I + s_0\\ -\frac12 I - s_0 & -\frac12  + |s|^2+Is_0\end{pmatrix},\]
which is defined for any $s\notin \{0\}\cup\ss$. For any $s\in\rho_S(T)$, the left $S$-resolvent is therefore given by
\[S_L^{-1}(s,T) = \frac1{2}|s|^{-2}(-1+|s|^2 + 2Is_0)^{-1}\begin{pmatrix}|s|^2 (I + 2 \overline{s}) + \overline{s} (-1 + 2 I s_0) &
-|s|^2 + \overline{s} (I + 2s_0) \\ |s|^2 - \overline{s} (I + 2 s_0) & |s|^2 (I + 2 \overline{s}) + \overline{s} (-1 + 2 Is_0)
\end{pmatrix}.\]
Since $\sigma_S(T)\cap\C_I = \{0,I,-I\}$, we choose $U_{\{0\}} = B_{1/2}(0)$ and $U_{\ss} =   B_2(0)\setminus B_{2/3}(0)$. For $s = \frac12 e^{I\varphi}\in \partial U_{\{0\}}(0)\cap\C_I$, we have
\[
S_L^{-1}(s,T) = 2e^{-I\varphi}\left(3I + 4\Re\left(e^{I\varphi}\right)\right)^{-1}
\begin{pmatrix}
I+e^{I\varphi}+2\cos(\varphi) & 2 + Ie^{I\varphi}+2I\cos\varphi\\ 
-2 - Ie^{I\varphi}+2I\cos\varphi & I+e^{I\varphi}+2\cos\varphi \end{pmatrix}
\]
and so
\begin{align*} E_{\{0\}} =& \frac1{2\pi}\int_{\partial(U_{\{0\}}\cap\C_I)} S_L^{-1}(s,T)\,ds_I\\
=&\frac{1}{2\pi} \int_{0}^{2 \pi} 2e^{-I\varphi}\left(3I + 4\Re\left(e^{I\varphi}\right)\right)^{-1}\cdot\\
&\phantom{\frac{1}{2\pi} \int_{0}^{2 \pi}}\cdot\begin{pmatrix}
I+e^{I\varphi}+2\cos(\varphi) & 2 + Ie^{I\varphi}-2I\cos\varphi\\ 
-2 - Ie^{I\varphi}+2I\cos\varphi & I+e^{I\varphi}+2\cos\varphi \end{pmatrix}\frac12 e^{I\varphi} I (-I) \, d\varphi\\
=& \frac12 \begin{pmatrix}
1 & - I\\ I & 1
\end{pmatrix}.
\end{align*}
A similar computation shows that
\[ E_{\ss} = \frac{1}{2\pi} \int_{\partial(U_{\ss}\cap\C_I)} S_L^{-1}(s,T)\, ds_I= \frac12\begin{pmatrix} 1 &I\\ -I&1\end{pmatrix}.\]
An easy computation shows that these matrices actually define projections on $\H^2$ with $E_{\{0\}} + E_{\ss}  = \id$. Moreover, we have $E_{\{0\}} v_1 = v_1$ and $E_{\ss} v_2 = 0$ as well as $E_{\{0\}}v_2 = 0$ and $E_{\ss}v_2 = v_2$. Thus, the invariant subspace $E_{\{0\}}V$ associated to the spectral set $\{0\}$ is the right linear span of $v_1$, which consist of all eigenvectors with respect to the real eigenvalues $0$ as $T(v_1)a = T(v_1)a = 0$ for all $a\in\H$. The invariant subspace $E_{\ss}$ associated to the spectral set $\ss$ consists of the right linear span of $v_2$. For $a\in\H\setminus\{0\}$, we have $T(v_2a) = T(v_2)a = v_2 Ia = (v_2a) (a^{-1}Ia)$. Thus, as $a^{-1}Ia\in\ss$, the subspace $E_{\ss}$ consists of all right eigenvectors associated to an eigenvalues in $\ss$. (This is true only because the associated subspace is one-dimensional! Otherwise the subspace consists of sums of eigenvectors associated to eigenvalues in the sphere, which do no have to be eigenvectors again.)

Finally, we can construct functions, which are left and right slice hyperholomorphic on $\sigma_S(T)$, but for which the $S$-functional calculi for left and right slice hyperholomorphic functions yield different operators: consider the function $f(s) = c_1 \chi_{U_{\{0\}}}(s) + c_2 \chi_{U_{\ss}}(s)$ such that $c_1$ or $c_2$ does not belong to $\C_I$.  Choose for instance $c_1 = J$ and $c_2 = 0$ for the sake of simplicity.  This function is a locally constant slice function on $U = U_{\{0\}}\cup U_{\ss}$ and thus left and right slice hyperholomorphic by \Cref{BHolSplit}. Then
\begin{gather*} 
\frac{1}{2\pi} \int_{\partial(U\cap\C_I)} S_L^{-1}(s,T)\,ds_I f(s) = \left(\frac{1}{2\pi} \int_{\partial(B_{1/2}(0)\cap\C_I)}S_L^{-1}(s,T) \,ds_I \right)J\\
 =\frac12  \begin{pmatrix} 1 & -I\\ I &1 \end{pmatrix}J = \frac12 \begin{pmatrix} J & -K\\ K &J\end{pmatrix},
\end{gather*}
but
\begin{gather*}
\frac{1}{2\pi}\int_{\partial(U\cap\C_I)} f(s)\,ds_I\,S_R^{-1}(s,T) = J \left(\frac{1}{2\pi}\int_{\partial(B_{1/2}(0)\cap\C_I)} \,ds_I\,S_R^{-1}(s,T)\right) \\
= \frac12 J\begin{pmatrix} 1 &-I\\ I & 1\end{pmatrix} = \frac12 \begin{pmatrix} J & K \\ -K & J\end{pmatrix}.
\end{gather*}
As pointed in \Cref{InconsRem}, the spectral projections cannot commute with arbitrary scalars because the respective invariant subspaces are not two-sided. Indeed, $-Jv_2 = (1, I) = v_1$, which does obviously not belong to $E_{\ss}V$. 
\end{example}

\printbibliography 

\end{document}